\documentclass[reqno, oneside, 11pt]{smfart}

\usepackage[T1]{fontenc}
\usepackage[utf8]{inputenc} 
\usepackage[english,french]{babel}

\usepackage{fullpage} 

\usepackage{mathtools, amssymb}  

\usepackage{dsfont, mathrsfs, empheq, tikz}
\usetikzlibrary{patterns} 

\usepackage{aliascnt}

\usepackage[sortcites=true, giveninits=true]{biblatex}
\usepackage{hyperref}   
\usepackage[capitalize]{cleveref}   

\numberwithin{equation}{section} 
\hypersetup{colorlinks=true,linkcolor={red!50!black},citecolor={blue!50!black},urlcolor={blue!99!black}}
\title{Asymptotics of Laplace and Stokes solutions in the gap between a sharp tip and a smooth boundary}

\author{Nicolas Roblet}
\address{Université de Montpellier, CNRS, Institut Montpelliérain Alexander Grothendieck, Montpellier, France}
\email{nicolas.roblet@umontpellier.fr}


\newcommand{\bigo}{\mathcal{O}}

\newcommand{\R}{\mathbb{R}}

\newcommand{\C}{\mathcal{C}}
\newcommand{\F}{\mathscr{F}}
\newcommand{\G}{\mathscr{G}}
\newcommand{\opsi}{\overline{\psi}\vphantom{\psi}}

\DeclareMathOperator{\odiv}{div}
\DeclareMathOperator{\sgn}{sgn}
\DeclareMathOperator{\dist}{dist}
\DeclareMathOperator{\diam}{diam}

\DeclareMathOperator{\supp}{supp}
\DeclareMathOperator*{\argmin}{arg\,min}

\theoremstyle{plain}
\newtheorem{theorem}{Theorem} 

\newtheorem{lemma}{Lemma}

\newaliascnt{corollary}{lemma}
\newtheorem{corollary}[corollary]{Corollary}
\aliascntresetthe{corollary}

\newaliascnt{property}{lemma}
\newtheorem{property}[property]{Property}
\aliascntresetthe{property}

\newaliascnt{proposition}{lemma}
\newtheorem{proposition}[proposition]{Proposition}
\aliascntresetthe{proposition}

\theoremstyle{remark} 
\newtheorem{remark}{Remark}

\crefname{theorem}{Theorem}{Theorems}
\crefname{lemma}{Lemma}{Lemmas}
\crefname{corollary}{Corollary}{Corollaries}
\crefname{section}{Section}{Sections}
\crefname{figure}{Figure}{Figures}
\crefname{table}{Table}{Tables}
\crefname{remark}{Remark}{Remarks}
\crefformat{equation}{\textup{(#2#1#3)}} 

\begin{document}

\selectlanguage{english}

\maketitle

\tableofcontents\label{toc}

\begin{abstract}
    We derive an explicit description of the asymptotic singular behavior of 2D Laplace and Stokes solutions as the distance between a Lipschitz sharp tip and a smooth boundary vanishes.
    While classical lubrication theory for smoother boundaries relies on the anisotropy of the gap to perform dimensional reduction, the isotropic scaling imposed by a corner singularity requires a different approach.
    We introduce a stability-based method that allows us to construct the leading-order asymptotics of these solutions in the gap.
    As an application, we determine the asymptotic behavior of the Stokes resistance matrix and the pressure field.
\end{abstract}

\section{Introduction}
\subsection{Context and motivation}

This paper investigates the singular interactions induced by the narrow-gap limit between two inclusions in a conducting medium, or analogously, two rigid particles suspended in a viscous fluid.
As the inter-particle distance vanishes, geometric confinement induces a blow-up of potential and velocity gradients, generating highly concentrated fluxes and singular forces.
At the macroscopic scale of a heterogeneous composite or a dense suspension, the accumulation of these localized effects alters the effective conductivity or rheology.
While homogenization theory is well established in the dilute regime, capturing these cumulative strong interactions in dense regimes remains fundamentally challenging.
Quantifying these pairwise interactions is thus an essential first step toward deriving effective macroscopic models.

Classical works on near-contact interactions focus primarily on smooth geometries, particularly the slow-motion dynamics of close spheres \cite{CooleyONeill1969,ONeillStewartson1967,ONeill1967a}.
The canonical configuration involves two proximate spheres or, equivalently, a solid sphere approaching a planar wall.
At first order, an opposing regular boundary reduces to its tangent plane, rendering the planar wall configuration generic.
This geometry is also of independent physical interest, modeling the sedimentation of a particle near a boundary and serving as a foundation for lubrication theory.
Formally, let $\Omega \subset \mathbb{R}^2_+$ be a bounded domain and $\mathcal{B}^d \subset \mathbb{R}^2_+$ a smooth inclusion (e.g., a disk) at distance $d$ from the flat boundary $\Gamma = \partial \mathbb{R}^2_+ \cap \partial \Omega$.
On the perforated domain $\mathcal{F}^d = \Omega \setminus \mathcal{B}^d$, the standard boundary value problem reads
\begin{equation*}
    \begin{cases}
        L u = 0 \qquad& \text{in } \mathcal{F}^d,\\
        u = u_* & \text{on } \partial \mathcal{B}^d,\\
        u = 0 & \text{on } \Gamma,
    \end{cases}
\end{equation*}
where $L$ is either the Laplace or the Stokes operator, and $u_*$ prescribes the Dirichlet boundary data on the inclusion.

In the Stokes framework, this ideal smoothness forces the hydrodynamic drag to diverge as $d^{-1}$.
This severe blow-up yields the well-known no-collision paradox \cite{Hillairet2007}: solid bodies driven by finite forces cannot reach contact in finite time.
However, experiments show that physical collisions do occur \cite{DavisZhaoGalvinWilson2003,DavisSerayssolHinch1986}, meaning that the paradox is an artifact of idealized geometry.
As established in \cite{GerardVaretHillairet2012}, finite-time contact can be mathematically restored by relaxing this regularity, for instance via Navier slip conditions or by considering $C^{1,a}$ boundaries with $a \in [0, 1)$.

The present paper extends this analysis to the Lipschitz regime by considering an inclusion featuring a sharp tip.
Beyond completing the mathematical picture for non-smooth boundaries, this setting is physically highly relevant.
In applied physics, the angularity of polygonal particles drastically alters the rheology and packing properties of suspensions \cite{vo_additive_2020,azema_packings_2013}.
In a dense suspension of such particles, the perfect alignment of two opposing tips is an event of probability zero.
Consequently, the generic pairwise interaction occurs between a tip and a flat edge.
We therefore focus on this canonical tip-to-plane configuration, conserving the flat lower boundary while replacing the smooth inclusion $\mathcal{B}^d$ with a tipped domain.
This specific geometry was recently considered in \cite{filippas_vector_2021}, where regularity arguments established that contact requires zero translational velocity.
Precise blow-up rates for the velocity gradient and pressure have recently been established in \cite{LiXuZhang2023, LiXu2025}, specifically for the case of regular, strictly convex geometries.
In contrast, our focus here is quantitative as we explicitly capture the leading-order asymptotic behavior of the solution.
Such a characterization enables, in the Stokes framework, the derivation of explicit asymptotics for related important quantities, notably the Stokes resistance matrix and the pressure field.
These results can also be of numerical utility for the simulation of dense suspensions, where resolving narrow-gap hydrodynamics remains a computational challenge (see, e.g., \cite{Lefebvre2009}).

To derive such asymptotics, the classical method relies on variational relaxation, as presented in \cite{bonheure2024,GerardVaretHillairet2012}.
The underlying principle is based on a scale separation within the narrow gap, where the energy concentrates.
For smooth boundaries, the gap domain is highly anisotropic and asymptotically mimics a thin channel.
Consequently, the governing energy functional associated with the operator $L$, such as the Dirichlet energy $\int_{\mathcal{F}^d} |\nabla u|^2$ for the Laplace problem, is dominated by vertical derivatives. 
This motivates relaxing the global energy minimization to a localized, one-dimensional problem, such as $\min \int_{\text{gap}} |\partial_{x_2} u|^2$.
Solving this reduced problem yields an explicit profile that captures the singular behavior of the exact solution.
In our Lipschitz geometry, however, this paradigm breaks down. 
A sharp tip, locally described by $x_2 = \alpha |x_1|$, lacks a horizontal tangent at the point of minimal distance.
As $d \to 0$, the fluid lacks a single dominant escape direction, rendering the variational relaxation completely ineffective.

To overcome this geometric obstruction, we introduce a new approach based on a spatial decomposition of the singular domain.
We approximate the gap regions on either side of the tip by unbounded angular sectors, on which the corresponding Laplace and Stokes problems can be solved explicitly.
By matching these local solutions, we construct a global composite ansatz.
Using stability estimates, we prove that this profile captures the exact leading-order behavior of the solution as $d \to 0$.
Because the constructed profiles are singular, we pay particular attention to the remainder, establishing that all committed errors remain non-singular.
Building on these results, we finally deduce the first-order asymptotic expansions for the pressure profile and the macroscopic hydrodynamic response of the system.

The paper is organized as follows.
The remainder of this section introduces the governing equations, notation, and main results.
\cref{sec_laplace} is devoted to the Laplace problem, detailing our stability-based approach and proving \cref{thm_lapl}.
\cref{sec_stokes} extends this methodology to the Stokes problem to establish \cref{thm_stokes}.
Building on this last result, \cref{sec_Stokes_matrix} establishes the first-order expansion of the Stokes resistance matrix in \cref{thm_stokes_matrices}, while \cref{sec_pressure} yields the analogous expansion for the pressure in \cref{thm_pression}.
Finally, in \cref{sec_var}, we highlight the limitations of standard variational methods in this geometry.

\subsection{Statement of the problem and of the main results}
Let $\Omega \subset \R^2$ be a smooth, bounded, and connected domain.
We assume that $\partial\Omega$ is flat in a neighborhood of the origin.
Specifically, we fix a Cartesian coordinate system $(O,e_1,e_2)$ such that $O\in\partial\Omega$ lies on this flat boundary segment, $e_1$ spans the tangent line to $\partial\Omega$ at $O$, and $e_2$ is the unit inward normal.
By a suitable choice of scale, we may assume without loss of generality that $[-2,2]\times\{0\}\subset \partial\Omega$, with the local interior of the domain $\Omega$ lying in the upper half-plane $\{(x_1,x_2)\in\R^2\mid x_2>0\}$.

Let $S^0$ be a simply-connected bounded inclusion with a $C^{0,1}$ boundary and a singular tip at its lowest extremity.
Locally near the tip, we assume the boundary of this inclusion consists of straight line segments within a neighborhood of length at least two.
Specifically, for some local radius $\rho>0$, the boundary can be parametrized in polar coordinates as
\begin{equation*}
    \partial {S^0}_{\mid\text{tip}} =\{(r,\theta)\mid r\in[0,\rho),~\theta\in\{\theta_R,\pi-\theta_L\}\},
\end{equation*} 
where the aperture angles $\theta_L,\theta_R\in(0,\pi)$ satisfy $\theta_R< \pi-\theta_L$.

We introduce the translated inclusion $S^d := d e_2 + S^0$, so that its tip is located at $d e_2$.
Here, $d>0$ denotes the distance between the inclusion and the flat boundary.
This quantity is the fundamental asymptotic parameter of the problem.
Assuming $d$ is sufficiently small and $S^d$ small enough so that $S^d\subset \Omega$, we define the perforated domain $\F^d=\Omega\setminus \overline{S^d}$.

We establish the primary geometric notation used throughout the paper.
Let $\C$ denote the closed sector with aperture angles $[\theta_R,\pi - \theta_L]$.
To analyze the region between the inclusion and the flat boundary, we define the local gap domain
\begin{equation*}
    G^d(\rho) := (B(0,\rho)\cap \{x_2>0\})\setminus (de_2 + \C).
\end{equation*}
For any radius $\rho \leq 2$, this domain is contained in $\F^d$, and its upper and lower boundaries coincide with those of $S^d$ and $\Omega$, respectively.
For $0<\rho_1<\rho_2$, we also define the annular gap region $A^d(\rho_1,\rho_2) := G^d(\rho_2) \setminus \overline{G^d(\rho_1)}$.
Next, the intersections of the horizontal axis with the lines characterizing the tip are given by $x_L = d y_L$ and $x_R = d y_R$ where
\begin{equation*}
    y_L=(\cot(\theta_L),0),\quad y_R = (-\cot(\theta_R),0).
\end{equation*}
For our subsequent analysis, fix $r_0 := 2\max(1, |y_L|, |y_R|)$ to ensure that, for $\rho>dr_0$, the annular region $A^d(dr_0, \rho)$ excludes the points $x_L, x_R$ and the tip.
Because the tip is excluded, this set decomposes into two connected components, which we denote by $A^d_L(dr_0, \rho)$ and $A^d_R(dr_0, \rho)$ for the left and right parts, respectively.
These geometric parameters are illustrated in \cref{fig:Ex_polygon_plane_sys}.
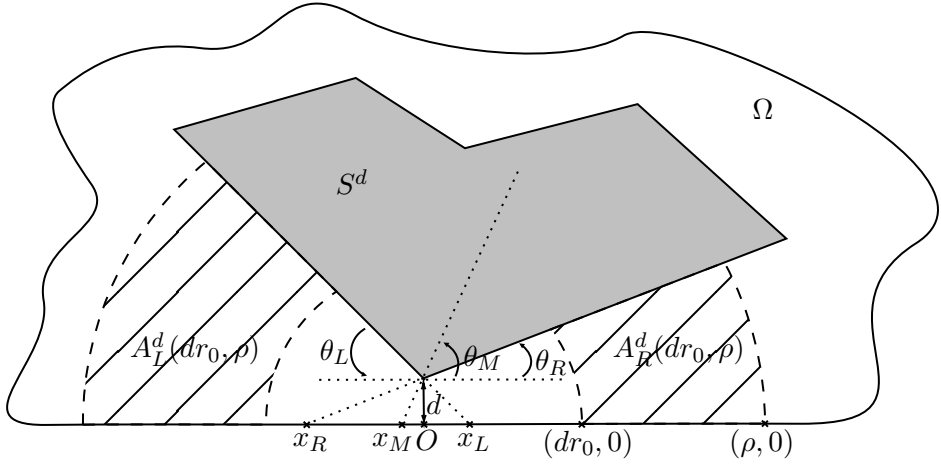
\begin{figure}[!ht]
    \centering
    \tikzset{
    pattern size/.store in=\mcSize, 
    pattern size = 5pt,
    pattern thickness/.store in=\mcThickness, 
    pattern thickness = 0.3pt,
    pattern radius/.store in=\mcRadius, 
    pattern radius = 1pt}
    \makeatletter
    \pgfutil@ifundefined{pgf@pattern@name@_tf5ramo2l}{
    \pgfdeclarepatternformonly[\mcThickness,\mcSize]{_tf5ramo2l}
    {\pgfqpoint{0pt}{0pt}}
    {\pgfpoint{\mcSize+\mcThickness}{\mcSize+\mcThickness}}
    {\pgfpoint{\mcSize}{\mcSize}}
    {
    \pgfsetcolor{\tikz@pattern@color}
    \pgfsetlinewidth{\mcThickness}
    \pgfpathmoveto{\pgfqpoint{0pt}{0pt}}
    \pgfpathlineto{\pgfpoint{\mcSize+\mcThickness}{\mcSize+\mcThickness}}
    \pgfusepath{stroke}
    }}
    \makeatother
    \tikzset{every picture/.style={line width=0.75pt}} 
    \begin{tikzpicture}[x=0.75pt,y=0.75pt,yscale=-0.75,xscale=0.75]
    \draw  [pattern=_tf5ramo2l,pattern size=22.5pt,pattern thickness=0.75pt,pattern radius=0pt, pattern color={rgb, 255:red, 0; green, 0; blue, 0}][dash pattern={on 4.5pt off 4.5pt}] (62.48,293.99) .. controls (89.16,293.77) and (83.77,293.79) .. (181.16,293.93) .. controls (182.41,293.94) and (183.67,293.94) .. (184.92,293.94) .. controls (184.92,293.6) and (184.92,293.26) .. (184.92,292.92) .. controls (184.92,256.2) and (203.65,223.85) .. (232.08,204.91) -- (144.91,117.88) .. controls (94.55,159.72) and (62.48,222.82) .. (62.48,293.42) .. controls (62.48,293.61) and (62.48,293.8) .. (62.48,293.99) -- cycle (396.25,293.5) .. controls (438.7,293.37) and (480.06,293.27) .. (518.69,293.33) .. controls (518.67,254.3) and (508.86,217.57) .. (491.57,185.46) -- (375.58,230.13) .. controls (388.57,247.69) and (396.25,269.4) .. (396.25,292.92) .. controls (396.25,293.11) and (396.25,293.3) .. (396.25,293.5) -- cycle ;
    \draw   (243.86,19.75) .. controls (294.33,47.08) and (388.5,54.5) .. (424,33.75) .. controls (459.5,13) and (694.92,125.06) .. (619.67,172.67) .. controls (544.41,220.28) and (650.87,293.74) .. (542.85,293.38) .. controls (434.83,293.03) and (301.05,294.11) .. (181.16,293.93) .. controls (61.28,293.76) and (97.13,293.76) .. (37.21,294.18) .. controls (-22.72,294.6) and (42.84,249.78) .. (36.18,210.83) .. controls (29.53,171.89) and (87.33,160.58) .. (68,133) .. controls (48.67,105.42) and (35.4,79.92) .. (46.67,70.67) .. controls (57.93,61.42) and (88.76,35.29) .. (141,42) .. controls (193.24,48.71) and (193.39,-7.58) .. (243.86,19.75) -- cycle ;
    \draw  [fill={rgb, 255:red, 0; green, 0; blue, 0 }  ,fill opacity=0.25 ] (245,62) -- (318,109) -- (433.5,79.5) -- (533,169.5) -- (290.25,263) -- (123.5,96.5) -- cycle ;
    \draw  [dash pattern={on 0.84pt off 2.51pt}]  (290.25,263) -- (321.38,293.75) ;
    \draw  [dash pattern={on 0.84pt off 2.51pt}]  (290.25,263) -- (211.75,293.75) ;
    \draw    (290.28,268) -- (290.56,289.92) ;
    \draw [shift={(290.58,291.92)}, rotate = 269.26] [color={rgb, 255:red, 0; green, 0; blue, 0 }  ][line width=0.75]    (4.37,-1.32) .. controls (2.78,-0.56) and (1.32,-0.12) .. (0,0) .. controls (1.32,0.12) and (2.78,0.56) .. (4.37,1.32)   ;
    \draw [shift={(290.25,266)}, rotate = 89.26] [color={rgb, 255:red, 0; green, 0; blue, 0 }  ][line width=0.75]    (4.37,-1.32) .. controls (2.78,-0.56) and (1.32,-0.12) .. (0,0) .. controls (1.32,0.12) and (2.78,0.56) .. (4.37,1.32)   ;
    \draw  [color={rgb, 255:red, 0; green, 0; blue, 0 }  ,draw opacity=1 ] (288.63,291.74) -- (292.69,295.89)(292.65,291.87) -- (288.67,295.76) ;
    \draw  [color={rgb, 255:red, 0; green, 0; blue, 0 }  ,draw opacity=1 ] (319.13,291.74) -- (323.19,295.89)(323.15,291.87) -- (319.17,295.76) ;
    \draw  [color={rgb, 255:red, 0; green, 0; blue, 0 }  ,draw opacity=1 ] (210.13,291.99) -- (214.19,296.14)(214.15,292.12) -- (210.17,296.01) ;
    \draw  [color={rgb, 255:red, 0; green, 0; blue, 0 }  ,draw opacity=1 ] (394.13,291.74) -- (398.19,295.89)(398.15,291.87) -- (394.17,295.76) ;
    \draw  [color={rgb, 255:red, 0; green, 0; blue, 0 }  ,draw opacity=1 ] (516.63,291.24) -- (520.69,295.39)(520.65,291.37) -- (516.67,295.26) ;
    \draw  [dash pattern={on 0.84pt off 2.51pt}]  (220.5,264) -- (383.25,263.75) ;
    \draw    (356.65,241.33) .. controls (365.57,248.93) and (362.74,257.59) .. (356.5,261.75) ;
    \draw [shift={(354.25,239.5)}, rotate = 34.25] [fill={rgb, 255:red, 0; green, 0; blue, 0 }  ][line width=0.08]  [draw opacity=0] (3.57,-1.72) -- (0,0) -- (3.57,1.72) -- cycle    ;
    \draw    (251.25,229.75) .. controls (238.53,236.92) and (239.29,254.14) .. (249.76,260.05) ;
    \draw [shift={(252.5,261.25)}, rotate = 198.1] [fill={rgb, 255:red, 0; green, 0; blue, 0 }  ][line width=0.08]  [draw opacity=0] (3.57,-1.72) -- (0,0) -- (3.57,1.72) -- cycle    ;
    \draw  [color={rgb, 255:red, 0; green, 0; blue, 0 }  ,draw opacity=1 ] (273.88,291.99) -- (277.94,296.14)(277.9,292.12) -- (273.92,296.01) ;
    \draw [line width=0.75]  [dash pattern={on 0.84pt off 2.51pt}]  (275.63,293.75) -- (354.63,122) ;
    \draw    (306.21,240.64) .. controls (314.68,246.69) and (314.83,254.42) .. (311.65,261.65) ;
    \draw [shift={(303.63,239)}, rotate = 29.6] [fill={rgb, 255:red, 0; green, 0; blue, 0 }  ][line width=0.08]  [draw opacity=0] (3.57,-1.72) -- (0,0) -- (3.57,1.72) -- cycle    ;
    \draw (283,295.98) node [anchor=north west][inner sep=0.75pt]  [font=\normalsize]  {$O$};
    \draw (290,271.07) node [anchor=north west][inner sep=0.75pt]  [font=\normalsize]  {$d$};
    \draw (310.83,298) node [anchor=north west][inner sep=0.75pt]  [font=\normalsize]  {$x_{L}$};
    \draw (200.83,298) node [anchor=north west][inner sep=0.75pt]  [font=\normalsize]  {$x_{R}$};
    \draw (253.83,298) node [anchor=north west][inner sep=0.75pt]  [font=\normalsize]  {$x_{M}$};
    \draw (367.83,295) node [anchor=north west][inner sep=0.75pt]  [font=\normalsize]  {$( dr_{0} ,0)$};
    \draw (492.83,295) node [anchor=north west][inner sep=0.75pt]  [font=\normalsize]  {$(\rho,0)$};
    \draw (229.58,120.9) node [anchor=north west][inner sep=0.75pt]  [font=\normalsize]  {$S^{d}$};
    \draw (92,230) node [anchor=north west][inner sep=0.75pt]  [font=\normalsize]  {$A^d_L(dr_0, \rho)$};
    \draw (415,230) node [anchor=north west][inner sep=0.75pt]  [font=\normalsize]  {$A^d_R(dr_0, \rho)$};
    \draw (508.33,73.9) node [anchor=north west][inner sep=0.75pt]  [font=\normalsize]  {$\Omega$};
    \draw (219,235) node [anchor=north west][inner sep=0.75pt]  [font=\normalsize]  {$\theta_L$};
    \draw (363,242) node [anchor=north west][inner sep=0.75pt]  [font=\normalsize]  {$\theta_R$};
    \draw (315,239) node [anchor=north west][inner sep=0.75pt]  [font=\normalsize]  {$\theta_M$};
    \end{tikzpicture}
    \caption[Example of polygon-plane interaction system]{\centering Example of tip-plane interaction system. The hatched area is $A^d(dr_0, \rho)=A^d_L(dr_0, \rho)\cup A^d_R(dr_0, \rho)$.}
    \label{fig:Ex_polygon_plane_sys}
\end{figure}

We first investigate the Laplace problem, a fundamental model for conductivity in heterogeneous media,
\begin{equation}\label{eq_lapl}
    \begin{cases}
        \Delta \phi_d =0\quad&\text{in }\F^d,\\
        \phi_d =\phi_*=1 \quad&\text{on }\partial S^d,\\
        \phi_d =0 \quad&\text{on }\partial \Omega.
    \end{cases}
\end{equation}
In the sequel, we extend $\phi_d$ to $\Omega$ by setting $\phi_d=1$ in the inclusion $S^d$.
By standard arguments, this problem admits a unique solution in $H^1_0(\Omega)$.
In particular, since $S^d$ introduces a corner in the boundary of $\F^d$, the regularity properties of the solution $\phi_d$ are governed by the general theory of elliptic boundary value problems in corner domains, as studied in \cite{Dauge1988,Grisvard2011}.
The following theorem establishes its leading-order asymptotic behavior.
\begin{theorem}[Asymptotic profile for the Laplace problem]\label{thm_lapl}
    We define the potential function $\widetilde{\phi}_d$ on $A^\text{tr}:=A^d(dr_0,1)$ by
    \begin{subequations}
    \begin{equation}\label{eq_1_th_Lapl}
        \widetilde{\phi}_d(x) = 
        \begin{cases}
            \dfrac{\pi - \arg(x - x_L)}{\theta_L} \quad& \text{in } A^\text{tr}_L, \\[0.5em]
            \dfrac{\arg(x - x_R)}{\theta_R}\quad & \text{in } A^\text{tr}_R.
        \end{cases}
    \end{equation}
    In the vanishing gap limit, this potential provides a sharp approximation of the exact solution $\phi_d$.
    Formally, as $d\to 0$, this potential satisfies
    \begin{empheq}[left=\empheqlbrace]{align}
        \|\widetilde{\phi}_d\|_{\dot{H}^1(A^\text{tr})}^2 &= (\theta_L^{-1} + \theta_R^{-1})|\log(d)| +\bigo(1), \label{eq_2_th_Lapl}\\
        \|\phi_d - \widetilde{\phi}_d\|^2_{\dot{H}^1(A^\text{tr})} &= \bigo(1),\label{eq_3_th_Lapl}\\
        \|\phi_d \|^2_{\dot{H}^1(\Omega\setminus A^\text{tr})}&= \bigo(1).\label{eq_4_th_Lapl}
    \end{empheq}
    In particular, these estimates yield the asymptotic expansion
    \begin{equation}
        \|\phi_d\|_{\dot{H}^1(\Omega)}^2 = (\theta_L^{-1} + \theta_R^{-1})|\log(d)| +\bigo(1).\label{eq_5_th_Lapl}
    \end{equation}
    Additionally, the potential remains bounded in $L^2$,
    \begin{equation}
        \|\phi_d\|_{L^2(\Omega)}^2 = \bigo(1).\label{eq_6_th_Lapl}
    \end{equation}
    \end{subequations}
\end{theorem}
\begin{remark}
    The annular gap $A^\text{tr}=A^d(dr_0,1)$ captures the transition between the thin-film scale $d$ at the tip and the far-field domain scale.
    This set governs the behavior of the problem, as the dominant contribution to the total energy concentrates within this region.
\end{remark}

Building upon our analysis of the Laplace equation, we turn to the Stokes equations.
Imposing a no-slip condition on the outer boundary $\partial \Omega$ and a prescribed rigid velocity $u_*$ on the inclusion boundary $\partial S^d$, we seek a velocity--pressure pair $(u_d, p_d)$ satisfying
\begin{equation}\label{eq_stokes}
    \begin{cases} 
        -\Delta u_d + \nabla p_d = 0 \quad &\text{in}~\F^d,\\
        \odiv u_d = 0 \quad &\text{in}~\F^d,\\
        u_d = u_* \quad&\text{on}~\partial S^d,\\
        u_d= 0  \quad&\text{on}~\partial\Omega.
    \end{cases}
\end{equation}
Because $u_*$ is a rigid motion, the compatibility condition $\int_{\partial S^d} u_* \cdot n = 0$ is automatically satisfied.
Classical theory \cite{Galdi2011} then guarantees the existence of a unique velocity field $u_d \in [H^1(\F^d)]^2$ and a corresponding pressure $p_d \in L^2(\F^d)$, unique up to an additive constant.
We extend $u_d$ to the entire domain $\Omega$ by setting $u_d = u_*$ inside the inclusion $S^d$, which yields a global velocity field in $[H^1(\Omega)]^2$.

By the linearity of the Stokes equations, the flow induced by an arbitrary rigid motion $u_*$ can be reconstructed from the solutions corresponding to the three canonical rigid motions:
\begin{equation}\label{eq_cano_motions}
    u_*^\|(x) = e_1,\qquad u_*^\perp(x) = e_2,\qquad  u_*^\circlearrowleft(x) = x^\perp.
\end{equation}
Our main result for the Stokes system describes the asymptotic behavior of the flow for these fundamental data.
\begin{theorem}[Asymptotic behavior of the Stokes flow]\label{thm_stokes}
    Let $A^\text{tr}:=A^d(dr_0,1)$ denote the transition gap region.
    We define the approximate stream functions $\widetilde{\psi}_d^\perp$, $\widetilde{\psi}_d^\|$ and $\widetilde{\psi}_d^\circlearrowleft$ on $A^\text{tr}$, corresponding respectively to the boundary data $u_*^\perp$, $u_*^\|$ and $u_*^\circlearrowleft$ by
    \begin{subequations}
    \begin{align}
        \widetilde{\psi}_d^\perp(x) &=\begin{cases}
                ~|x-x_L|~f_L^\perp\big(\arg(x-x_L)\big)\quad& \text{in } A^\text{tr}_L, \\[0.5em]
                ~|x-x_R|~f_R^\perp\big(\arg(x-x_R)\big)\quad& \text{in } A^\text{tr}_R,
            \end{cases} \label{eq_1_th_stokes} \\[0.5em]
        \widetilde{\psi}_d^\|(x) &=\begin{cases}
                ~|x-x_L|~f_L^\|\big(\arg(x-x_L)\big)~\quad& \text{in } A^\text{tr}_L, \\[0.5em]
                ~|x-x_R|~f_R^\|\big(\arg(x-x_R)\big)~\quad& \text{in } A^\text{tr}_R,
            \end{cases}\label{eq_2_th_stokes}\\[0.5em]
        \widetilde{\psi}_d^\circlearrowleft(x) &=\begin{cases}
                \frac{1}{2}|x-x_L|^2f_L^\circlearrowleft\big(\arg(x-x_L)\big)\quad& \text{in } A^\text{tr}_L, \\[0.5em]
                \frac{1}{2}|x-x_R|^2f_R^\circlearrowleft\big(\arg(x-x_R)\big)\quad& \text{in } A^\text{tr}_R,
            \end{cases}\label{eq_3_th_stokes}
    \end{align}
    where the angular functions $f_L$ and $f_R$ are specified in \cref{lem_expr_psiL_psiR}.
    Setting $\widetilde{u}_d := \nabla^\perp \widetilde{\psi}_d$, in the vanishing gap limit $d\to 0$, these approximate velocity fields satisfy
    \begin{equation}\label{eq_4_th_stokes}
        \begin{cases}
        \|\widetilde{u}_d^\perp\|_{\dot{H}^1(A^\text{tr})}^2 & = 2\left[\frac{\theta_L + \sin(\theta_L)\cos(\theta_L)}{\theta_L^2 - \sin^2(\theta_L)} + \frac{\theta_R + \sin(\theta_R)\cos(\theta_R)}{\theta_R^2 - \sin^2(\theta_R)}\right]|\log(d)| + \bigo(1),\\[0.5em]
        \|\widetilde{u}_d^\|\|_{\dot{H}^1(A^\text{tr})}^2 & = 2\left[\frac{\theta_L - \sin(\theta_L)\cos(\theta_L)}{\theta_L^2 - \sin^2(\theta_L)} + \frac{\theta_R - \sin(\theta_R)\cos(\theta_R)}{\theta_R^2 - \sin^2(\theta_R)}\right]|\log(d)| +\bigo(1),\\[0.5em]
        \|\widetilde{u}_d^\circlearrowleft\|_{\dot{H}^1(A^\text{tr})}^2 & = \bigo(1),
        \end{cases}
    \end{equation}
    and, for any elementary boundary condition,
    \begin{empheq}[left=\empheqlbrace]{align}
        \|u_d - \widetilde{u}_d\|^2_{\dot{H}^1(A^\text{tr})} &= \bigo(1),\label{eq_5_th_stokes}\\
        \|u_d \|^2_{\dot{H}^1(\Omega\setminus A^\text{tr})}&= \bigo(1).\label{eq_6_th_stokes}
    \end{empheq}
    In particular, these estimates yield the asymptotic expansion
    \begin{equation}\label{eq_7_th_stokes}
        \|u_d\|^2_{\dot{H}^1(\Omega)} = \|\widetilde{u}_d\|^2_{\dot{H}^1(A^\text{tr})} + \bigo(1)
    \end{equation}
    Additionally, the potential remains bounded in $L^2$,
    \begin{equation}
        \|u_d\|^2_{L^2(\Omega)}= \bigo(1).\label{eq_8_th_stokes}
    \end{equation}
    \end{subequations}
\end{theorem}
This result provides the necessary tools to derive the asymptotic expansions of the Stokes resistance matrix and the pressure, established in \cref{thm_stokes_matrices} and \cref{thm_pression}, respectively.

\subsection*{Acknowledgments}
I would like to thank my PhD advisors, David Gérard-Varet and Matthieu Hillairet, for their guidance and support.
This work was supported by the French National Research Agency through the PEPR Maths-Vives program, as part of the "Complexflows" project (ANR-23-EXMA-0004).

\section{Laplace problem -- Proof of Theorem 1}\label{sec_laplace}

The standard approach for characterizing singular behavior in near-contact problems relies on a variational reduction of the energy functional.
By exploiting a separation of scales, this method relaxes the governing equations to yield an explicitly solvable limit problem.
However, as shown in \cref{sec_var}, this approach breaks down for Lipschitz inclusions.
To overcome this limitation, we introduce an alternative framework based on a stability argument.

Guided by the ansatz that the energy concentrates in the macroscopic gap region, our strategy proceeds in two steps.
We first construct an explicit local solution to the original problem within a transition region between the thin-film and far-field scales.
Next, we extend this local profile to a global candidate and employ a stability argument to prove it captures the leading-order behavior of the exact solution.

We first present this method for the scalar Laplace problem.
While the extension to the full Stokes system introduces technical complexities, the underlying geometric decomposition and core stability argument remain structurally identical.

\subsection{Construction of the local profile}\label{sec_const_prof_laplace}
We construct a local solution to the Laplace problem on the annular region $A^d(\lambda dr_0, 2)$, where $\lambda \in (\frac{1}{2}, 1)$ is a fixed parameter.
This domain is chosen to strictly contain $A^d(dr_0, 1)$ to accommodate the spatial cutoff functions required subsequently.
We detail this construction to clarify the emergence of the annulus $A^d(\lambda dr_0, 2)$.
Recall that this region captures the transition between the singular scale $d$ and the far-field domain while strictly excluding the tip and the intersection points $x_L$, $x_R$.

Our objective is to construct a profile that locally solves the Laplace problem near the tip, namely in $G^d(1)$.
To resolve the near-tip singularity, we introduce the blow-up coordinate $y = x/d$.
This scaling translates the localized singularity into a far-field problem on the expanded domain $d^{-1}G^d(1)=G^1(d^{-1})$.
As $d \to 0$, this rescaled domain expands to the unbounded region $G^1(\infty) = \{x_2>0\} \setminus (e_2+\C)$.
By construction of $y_L$ and $y_R$, we have $G^1(\infty) = (\C_L + y_L)\cup(\C_R + y_R)$, where
\begin{equation*}
    \C_L :=\left\{(r,\theta) \mid ~r>0,~ \theta\in(\pi-\theta_L,\pi)\right\}, \quad \C_R :=\left\{(r,\theta)\mid ~r>0,~ \theta\in(0,\theta_R)\right\}.
\end{equation*}
This geometric decomposition is motivated by the fact that the Laplace problem admits explicit solutions on these unbounded sectors.
Specifically, the boundary value problems
\begin{equation*}
    \begin{cases}
        \Delta \phi_L =0\quad&\text{in }\C_L,\\
        \phi_L = 1 \quad&\text{on }\{\theta = \pi-\theta_L\},\\
        \phi_L = 0 \quad&\text{on }\{\theta = \pi\},
    \end{cases}\qquad
    \begin{cases}
        \Delta \phi_R =0\quad&\text{in }\C_R,\\
        \phi_R = 1 \quad&\text{on }\{\theta = \theta_R\},\\
        \phi_R = 0 \quad&\text{on }\{\theta = 0\},
    \end{cases} 
\end{equation*}
admit the explicit solutions
\begin{equation}\label{eq:expr_phiLR_loc}
    \phi_L(y)=\frac{\pi-\theta}{\theta_L}, \qquad \phi_R(y) = \frac{\theta}{\theta_R}.
\end{equation}
We glue these solutions to define the piecewise local profile
\begin{equation*}
    \widetilde{\phi}_r(y) := \begin{cases}
        \phi_L(y-y_L)\quad &\text{in } A^1_L(\lambda r_0,2d^{-1}),\\
        \phi_R(y-y_R)\quad &\text{in } A^1_R(\lambda r_0,2d^{-1}).
    \end{cases}
\end{equation*}

\begin{lemma}[Local profile]\label{lem_jonct_lapl}
    The function $\widetilde{\phi}_r$ is well-defined, smooth on $A^1(\lambda r_0,2d^{-1})$, and satisfies
    \begin{equation*}
        \begin{cases}
            \Delta \widetilde{\phi}_r = 0 \quad& \text{in } A^1(\lambda r_0,2d^{-1}), \\
            \widetilde{\phi}_r = 1 \quad& \text{on } \partial [d^{-1} S^d] \cap \partial A^1(\lambda r_0,2d^{-1}), \\
            \widetilde{\phi}_r = 0 \quad& \text{on } \partial [d^{-1}\Omega] \cap \partial A^1(\lambda r_0,2d^{-1}).
        \end{cases}
    \end{equation*}
\end{lemma}
\begin{proof}
    By our choice of $r_0=2\max(1, |y_L|, |y_R|)$ and $\lambda\in(\frac{1}{2},1)$, the rescaled tip $(0,1)$ is contained in the gap $G^1(\lambda r_0)$.
    Consequently, the region $A^1(\lambda r_0,2d^{-1})$ splits into two disjoint connected components corresponding to the left and right branches, ensuring that $\widetilde{\phi}_r$ is well-defined.
    Furthermore, because the singular points $y_L$ and $y_R$ of the shifted profiles $\phi_L(\cdot-y_L)$ and $\phi_R(\cdot-y_R)$ are excluded from $A^1(\lambda r_0,2d^{-1})$, the resulting function $\widetilde{\phi}_r$ is smooth. 
    Finally, the respective translations by $y_L$ and $y_R$ guarantee that $\widetilde{\phi}_r$ satisfies the boundary condition on $d^{-1} S^d$.
    Both shifted profiles vanish on the lower boundary, thus satisfying the homogeneous Dirichlet boundary condition.
\end{proof}
Returning to the original coordinates $x=dy$, we define the gap profile on $A^d(\lambda dr_0,2)$ by
\begin{equation*}
    \widetilde{\phi}_\text{gap}(x) := \widetilde{\phi}_r\left(\frac{x}{d}\right) =\begin{cases}
            \frac{\pi -\arg(x-x_L)}{\theta_L}\quad &\text{in } A^d_L(\lambda d r_0,2),\\[0.5em]
            \frac{\arg(x-x_R)}{\theta_R}\quad &\text{in } A^d_R(\lambda d r_0,2).
        \end{cases}
\end{equation*}
Given $\mu_1 > 0$, for any $x \in A^d_{L/R}(\lambda d r_0,2)$ satisfying $ |x-x_{L/R}|> \mu_1$, the profile satisfies the pointwise bounds
\begin{equation}\label{eq_bound_profil_gap_phi}
    |\widetilde{\phi}_{\text{gap}}(x)| \leq C(\theta_L,\theta_R), \qquad |\nabla\widetilde{\phi}_{\text{gap}}(x)| \leq C(\theta_L,\theta_R)\mu_1^{-1}.
\end{equation}
Furthermore, \cref{lem_jonct_lapl} ensures that $\widetilde{\phi}_{\text{gap}}$ solves 
\begin{equation*}
    \begin{cases}
        \Delta \widetilde{\phi}_\text{gap} = 0 \quad& \text{in }A^d(\lambda dr_0,2), \\
        \widetilde{\phi}_\text{gap} = 1 \quad& \text{on } \partial S^d \cap \partial [A^d(\lambda dr_0,2)], \\
        \widetilde{\phi}_\text{gap} = 0 \quad& \text{on } \partial \Omega \cap \partial [A^d(\lambda dr_0,2)].
    \end{cases}
\end{equation*}
Thus, $\widetilde{\phi}_\text{gap}$ provides a suitable approximation in the transition region; we next quantify its energy.
\begin{lemma}[Local profile energy]\label{lem_estim_tildepsir_lapl}
    As $d \to 0$, the constructed function $\widetilde{\phi}_{\text{gap}}$ satisfies
    \begin{equation*}
        \|\widetilde{\phi}_{\text{gap}}\|_{\dot{H}^1(A^\text{tr})}^2 = (\theta_L^{-1} + \theta_R^{-1}) |\log(d)| + \bigo(1) .
    \end{equation*}
\end{lemma}
\begin{proof}
    Observe first that $A^\text{tr}:=A^d(dr_0,1)\subset A^d(\lambda dr_0,2)$.
    Using $\nabla \arg(x) = \frac{x^\perp}{|x|^2}$, direct computation yields
    \begin{equation*}
        \|\nabla\widetilde{\phi}_{\text{gap}}\|_{L^2(A^\text{tr})}^2 = \int_{A_L^\text{tr}} \frac{1}{\theta_L^2|x-x_L|^2} + \int_{A_R^\text{tr}} \frac{1}{\theta_R^2|x-x_R|^2}.
    \end{equation*}
    By symmetry, it suffices to estimate the second term.
    We decompose the shifted domain of integration into the disjoint union
    \begin{equation*}
        A^\text{tr}_R - x_R = A_\text{main} \cup A_\text{rem},
    \end{equation*}
    where $A_\text{main}$ is the angular annulus $\{(r,\theta) \mid r\in(dr_1,1), \theta\in(0,\theta_R)\}$.
    Here, the radius $r_1>0$ depends only on $\theta_L$ and $\theta_R$, and is chosen large enough so that $A_\text{main} \subset A^\text{tr}_R - x_R$.
    Evaluating the integral over the main annulus gives
    \begin{equation*}
        \int_{A_\text{main}}|\nabla\phi_R|^2 = \int^{\theta_R}_{0}\int_{r_1}^{d^{-1}}\frac{1}{\theta_R^2r^2}rdrd\theta = \theta_R^{-1} \left|\log\left(d\right)\right|+C(\theta_L,\theta_R),
    \end{equation*}
    where $C(\theta_L,\theta_R)>0$ denotes a generic constant depending only on $\theta_L$ and $\theta_R$.
    The remainder $A_\text{rem}$ consists of an inner component contained within an annulus with radii proportional to $d$, and an outer component contained within an annulus with $d$-independent radii.
    A similar direct calculation bounds the contribution of this remainder by
    \begin{equation*}
        \int_{A_\text{rem}}|\nabla\phi_R|^2 \leq C(\theta_L,\theta_R).
    \end{equation*}
    Summing these estimates yields the desired asymptotic behavior.
\end{proof}

\subsection{Error estimation}
To compare $\widetilde{\phi}_{\text{gap}}$ with the exact global solution $\phi_d$ via a stability argument, we must extend the local profile to the entire domain $\F^d$.
We introduce smooth radial cutoff functions $\chi_1,\chi_2\in C_c^\infty(\R_+)$ satisfying
\begin{equation}\label{eq_def_chi_ext}
    \mathds{1}_{[0,\lambda)} \leq \chi_1 \leq \mathds{1}_{[0,1)},\qquad \mathds{1}_{[0,1)} \leq \chi_2 \leq \mathds{1}_{[0,2)},
\end{equation}
and define the global approximation $\widetilde{\phi}_d$ on $\F^d$ by
\begin{equation}\label{eq_def_phi_tilde_d}
    \widetilde{\phi}_d(x)=\left[1-\chi_1\left(\frac{|x|}{dr_0}\right)\right]\chi_2(|x|)\widetilde{\phi}_\text{gap}(x).
\end{equation}
Function $\widetilde{\phi}_d$ is smooth on $\F^d$ and coincides with $\widetilde{\phi}_\text{gap}$ on $A^\text{tr}=A^d(dr_0,1)$.
Consequently, it satisfies
\begin{equation}\label{eq:syst_local_phid}
    \begin{cases}
        \Delta \widetilde{\phi}_d = 0 \quad& \text{in } A^\text{tr}, \\
        \widetilde{\phi}_d = 1 \quad& \text{on } \partial S^d \cap \partial A^\text{tr}, \\
        \widetilde{\phi}_d = 0 \quad& \text{on } \partial \Omega.
    \end{cases}
\end{equation}
Away from $A^\text{tr}$, $\widetilde{\phi}_d$ transitions smoothly to zero near the singular tip within the inner annulus $A^\text{in}=A^d(\lambda dr_0,dr_0)$ and in the far field within the outer annulus $A^\text{out}=A^d(1,2)$.
In particular, it vanishes on $\F^d\setminus A^d(\lambda dr_0,2)$.
The following lemma bounds the energy of this extension.
\begin{lemma}[Lifting energy]\label{lem_estim_ext_lapl}
    As $d\to 0$, the energy of the extension $\widetilde{\phi}_d$ outside $A^\text{tr}$ satisfies
    \begin{equation*}
        \|\widetilde{\phi}_d\|_{\dot{H}^1(\F^d\setminus A^\text{tr})} = \bigo(1).
    \end{equation*}
\end{lemma}
\begin{proof}
    Since $\supp(\widetilde{\phi}_d) \setminus A^\text{tr} \subset A^\text{in} \cup A^\text{out}$, we may estimate the energy contributions from the inner and outer transition annuli separately.
    In $A^\text{out}$, the inner cutoff vanishes identically, reducing the function to $\widetilde{\phi}_d=\chi_2(|x|)\widetilde{\phi}_\text{gap}$.
    Since the measure $|A^\text{out}|$ and the functions $\chi_2, \widetilde{\phi}_\text{gap}$ (along with their gradients) are uniformly bounded, it follows that $\|\widetilde{\phi}_d\|_{\dot{H}^1(A^\text{out})} =\bigo(1)$.
    In $A^\text{in}$, the outer cutoff satisfies $\chi_2=1$.
    Differentiating yields
    \begin{equation}\label{eq_estim_nab_phid}
        \big|\nabla \widetilde{\phi}_d(x)\big| \leq C\left[\frac{1}{dr_0} \big|\widetilde{\phi}_\text{gap}(x)\big| + \big|\nabla\widetilde{\phi}_\text{gap}(x)\big|\right].
    \end{equation}
    Since $|x| > \lambda dr_0$ in $A^{\mathrm{in}}$, \cref{eq_bound_profil_gap_phi} ensures that $|\nabla\widetilde{\phi}_{\text{gap}}(x)| \leq C (\lambda d r_0)^{-1} = \bigo(d^{-1})$.
    The full gradient therefore scales as $\bigo(d^{-1})$.
    Observing that $|A^\text{in}| = \bigo(d^2)$, integrating this pointwise bound gives
    \begin{equation*}
        \|\widetilde{\phi}_d\|_{\dot{H}^1(A^\text{in})} \leq \bigo(d^{-1})|A^\text{in}|^{1/2} = \bigo(1).
    \end{equation*}
\end{proof}
\begin{lemma}[Approximation error]\label{lem_estim_error_approx_lapl}
    As $d \to 0$, the approximation error satisfies
    \begin{equation*}
        \|\phi_d-\widetilde{\phi}_d\|_{\dot{H}^1(\F^d)} = \bigo(1).
    \end{equation*}
\end{lemma}
\begin{proof}
    Let $\phi^\text{err}=\phi_d-\widetilde{\phi}_d$.
    Anticipating its explicit construction, we postulate the existence of a boundary lifting $\phi_\text{BC} \in H^1(\F^d)$ such that $\phi^\text{err}_0 := \phi^\text{err} - \phi_\text{BC} \in H^1_0(\F^d)$.
    The weak formulation for $\phi_d$ ensures
    \begin{equation*}
        \int_{\F^d}\nabla \phi^\text{err}_0\cdot \nabla v = - \int_{\F^d}\nabla \widetilde{\phi}_d\cdot \nabla v - \int_{\F^d}\nabla \phi_\text{BC} \cdot \nabla v, \qquad \forall v \in H^1_0(\F^d).
    \end{equation*}
    Choosing $v = \phi^\text{err}_0$, it follows that $\|\phi^\text{err}_0\|_{\dot{H}^1(\F^d)} = \bigo(1)$ provided the right-hand side defines a bounded linear functional on $H^1_0(\F^d)$ with a norm uniformly bounded in $d$.

    We first estimate the term involving $\widetilde{\phi}_d$.
    Integration by parts yields
    \begin{equation*}
        \int_{\F^d} \nabla \widetilde{\phi}_d \cdot \nabla v = -\int_{\F^d} \Delta \widetilde{\phi}_d v, \qquad \forall \phi \in H^1_0(\F^d).
    \end{equation*}
    By \cref{eq:syst_local_phid}, $\Delta \widetilde{\phi}_d$ vanishes identically on the transition annulus $A^\text{tr}$, restricting its support to $A^\text{in} \cup A^\text{out}$.
    In the inner region $A^\text{in}=A^d(\lambda dr_0,dr_0)$, each derivative of the cutoff $\chi_1$ or $\widetilde{\phi}_\text{gap}$ introduces a factor of order $d^{-1}$.
    Arguing as in \cref{eq_estim_nab_phid}, we deduce $|\Delta\widetilde{\phi}_d| = \bigo(d^{-2})$.
    Since $|A^\text{in}| = \bigo(d^2)$, it follows that $\| \Delta \widetilde{\phi}_d \|_{L^2(A^\text{in})} = \bigo(d^{-1})$.
    Furthermore, because $v\in H^1_0(\F^d)$ vanishes on $\partial \Omega \cap \partial A^\text{in}$ and $\diam(A^\text{in}) =\bigo(d)$, the local Poincaré inequality implies $\| v\|_{L^2(A^\text{in})}\leq C d \| v\|_{\dot{H}^1(A^\text{in})}$.
    Thus,
    \begin{equation*}
        \left|\int_{A^\text{in}} \Delta \widetilde{\phi}_d  v\right| \leq \| \Delta \widetilde{\phi}_d \|_{L^2(A^\text{in})} \| v\|_{L^2(A^\text{in})}\leq C\| v\|_{\dot{H}^1(A^\text{in})}.
    \end{equation*}
    In the far-field region $A^\text{out}=A^d(1,2)$, the inner cutoff $\chi_1$ vanishes.
    Here, \cref{eq_bound_profil_gap_phi} ensures $\| \Delta \widetilde{\phi}_d \|_{L^2(A^\text{out})} \leq C$, and a standard Poincaré inequality on $A^\text{out}$ yields
    \begin{equation*}
        \left|\int_{A^\text{out}} \Delta \widetilde{\phi}_d  v\right| \leq \| \Delta \widetilde{\phi}_d \|_{L^2(A^\text{out})} \| v\|_{L^2(A^\text{out})} \leq C\| v\|_{\dot{H}^1(A^\text{out})}.
    \end{equation*}

    We next construct a boundary lifting $\phi_\text{BC}$ satisfying
    \begin{equation}\label{eq_unif_bound_lift_lapl}
        \|\phi_\text{BC}\|_{\dot{H}^1(\F^d)} =\bigo(1).
    \end{equation}
    By definition, $\widetilde{\phi}_d = 1$ on $\partial S^d \cap \partial A^\text{tr}$ and $\widetilde{\phi}_d = 0$ on $\partial \Omega \cap \partial A^\text{tr}$.
    The required boundary correction naturally decomposes into an inner contribution on $\partial \F^d \cap \partial G^d(dr_0)$ and an outer far-field contribution on $\partial \F^d \setminus \partial G^d(1)$.
    Accordingly, we split the lifting as $\phi_\text{BC} := \phi_\text{BC}^\text{tip} + \phi_\text{BC}^\text{far}$.

    The inner lifting $\phi_\text{BC}^\text{tip}$ must satisfy $\phi_\text{BC}^\text{tip}=1-\widetilde{\phi}_d$ on $\partial S^d \cap \partial G^d(dr_0)$ and vanish on the remainder of $\partial\F^d$.
    Let $\chi_\text{tip} \in C_c^\infty(\R_+)$ be a cutoff function satisfying $\mathds{1}_{[0,1)} \leq \chi_\text{tip} \leq \mathds{1}_{[0,2)}$.
    For $x \in \F^d$, we define
    \begin{equation*}
        \phi_\text{BC}^\text{tip}(x) := \chi_\text{tip}\left(\frac{\dist(x,S^d)}{d/4}\right)\chi_\text{tip}\left(\frac{|x|}{dr_0}\right)\big[1-\widetilde{\phi}_d(x)\big].
    \end{equation*}
    Since $x \mapsto \dist(x,S^d)$ is Lipschitz continuous and $\chi_\text{tip}$ is constant near the origin (avoiding the singularity of $|x|$), the smoothness of $\widetilde{\phi}_d$ guarantees that $\phi_\text{BC}^\text{tip} \in H^1(\F^d)$.
    The first factor restricts the support to a $(d/2)$-tubular neighborhood of $S^d$.
    Because this layer remains strictly separated from the bottom boundary $\partial \Omega$, we have $\phi_\text{BC}^\text{tip} = 0$ on $\partial \Omega$.
    The second cutoff confines the support to $G^d(2dr_0)$.
    Finally, on the intermediate top boundary $\partial S^d \cap \partial A^d(dr_0,2dr_0)$, \cref{eq:syst_local_phid} gives $\widetilde{\phi}_d = 1$.
    Consequently, $\phi_\text{BC}^\text{tip}$ satisfies the exact trace conditions desired.
    To estimate its energy contribution, we observe that the scaling arguments from the proof of \cref{lem_estim_ext_lapl} gives that each differentiation introduces a factor of order at most $d^{-1}$.
    Because $\supp(\widetilde{\phi}_d)\subset\{|x|>\lambda dr_0\}$, \cref{eq_bound_profil_gap_phi} implies $|\nabla \phi_\text{BC}^\text{tip}| = \bigo(d^{-1})$. Integrating over the support, whose measure is $\bigo(d^2)$, compensates for this singularity, yielding $\|\phi_\text{BC}^\text{tip}\|_{\dot{H}^1(\F^d)} = \bigo(1)$.

    We define the outer lifting on $\F^d$ by
    \begin{equation*}
        \phi_\text{BC}^\text{far}(x) := \chi_\text{tip}\big(\dist(x,S^d)\big)\big(1-\chi_\text{tip}(2|x|)\big)\big[1-\widetilde{\phi}_d(x)\big].
    \end{equation*}
    This function satisfies $\phi_\text{BC}^\text{far}= 1 - \widetilde{\phi}_d$ on $\partial S^d \setminus \partial G^d(1)$ and vanishes elsewhere on $\partial\F^d$.
    In the far-field region, the inner cutoff $\chi_1$ vanishes identically.
    No singular scaling appears, and both the gradient and the measure of the support are of order one, yielding $\|\phi_\text{BC}^\text{far}\|_{\dot{H}^1(\F^d)}= \bigo(1)$.
    Summing these components yields a valid lifting $\phi_\text{BC} \in H^1(\F^d)$.

    It follows that the linear map $v \mapsto \int_{\F^d} \nabla \phi_\text{BC} \cdot \nabla v$ is bounded independently of $d$ in $H^{-1}(\F^d)$.
    Returning to the weak formulation, we obtain $\|\phi^\text{err}_0\|_{\dot{H}^1(\F^d)} = \bigo(1)$, and the desired bound on the approximation error $\phi^\text{err}$ follows by the triangle inequality.
\end{proof}
\begin{proof}[Proof of \cref{thm_lapl}]
    By definition \cref{eq_def_phi_tilde_d}, the approximation $\widetilde{\phi}_d$ coincides with the gap profile $\widetilde{\phi}_\text{gap}$ on the transition region $A^\text{tr}$, yielding \cref{eq_1_th_Lapl}.
    Consequently, \cref{lem_estim_tildepsir_lapl} establishes \cref{eq_2_th_Lapl}.

    Since $A^\text{tr}\subset \F^d$, the local error bound \cref{eq_3_th_Lapl} follows immediately from \cref{lem_estim_error_approx_lapl}.
    To establish \cref{eq_4_th_Lapl}, we split the outer domain $\Omega \setminus A^\text{tr}$ into $S^d$ and $\F^d \setminus A^\text{tr}$ to obtain
    \begin{equation*}
        \|\phi_d\|_{\dot{H}^1(\Omega\setminus A^\text{tr})} \leq \|\phi_d\|_{\dot{H}^1(S^d)} + \|\widetilde{\phi}_d\|_{\dot{H}^1(\F^d\setminus A^\text{tr})} +\|\phi_d - \widetilde{\phi}_d\|_{\dot{H}^1(\F^d\setminus A^\text{tr})}.
    \end{equation*}
    Since $\phi_d = 1$ on $S^d$, the first term on the right-hand side vanishes.
    The second and third terms are controlled by \cref{lem_estim_ext_lapl} and \cref{lem_estim_error_approx_lapl}, respectively.

    To prove \cref{eq_5_th_Lapl}, we expand 
    \begin{equation*}
        \|\phi_d\|^2_{\dot{H}^1(\F^d)}  = \|\widetilde{\phi}_d\|^2_{\dot{H}^1(\F^d)} + 2 \int_{\F^d}\nabla\widetilde{\phi}_d\cdot\nabla(\phi_d - \widetilde{\phi}_d) +\|\phi_d - \widetilde{\phi}_d\|^2_{\dot{H}^1(\F^d)}.
    \end{equation*}
    The first term admits a uniform bound by \cref{lem_estim_ext_lapl,lem_estim_tildepsir_lapl}, and the third by \cref{lem_estim_error_approx_lapl}.
    For the cross-term, a direct application of the Cauchy--Schwarz inequality yields a suboptimal bound of $\bigo(\sqrt{|\log d|})$.
    To obtain a $\bigo(1)$ estimate, we recall the notation of the preceding proof and rewrite
    \begin{equation*}
        \int_{\F^d}\nabla\widetilde{\phi}_d\cdot\nabla\phi^\text{err} = \int_{\F^d}\nabla\phi_d \cdot\nabla\phi^\text{err}  - \|\phi^\text{err}\|^2_{\dot{H}^1(\F^d)}.
    \end{equation*}
    Since the $\dot{H}^1$-norm of $\phi^\text{err}$ is bounded, it suffices to estimate
    \begin{equation*}
         \int_{\F^d}\nabla\phi_d \cdot\nabla\phi^\text{err} = \int_{\F^d}\nabla\phi_d \cdot\nabla\phi^\text{err}_0 + \int_{\F^d}\nabla\phi_d \cdot\nabla\phi_\text{BC}.
    \end{equation*}
    The first integral vanishes by the weak formulation of the Laplace problem \cref{eq_lapl}.
    The second involves the lifting $\phi_\text{BC}$, which remains of order one by \cref{eq_unif_bound_lift_lapl} and supported in $\F^d \setminus A^d(2dr_0, \frac{1}{2})$.
    On this support, we have
    \begin{equation*}
        \|\phi_d\|_{\dot{H}^1(\F^d\setminus A^d(2dr_0,\frac{1}{2}))}^2 \leq \|\phi_d\|^2_{\dot{H}^1(\F^d\setminus A^\text{tr})} + \|\widetilde{\phi}_d\|^2_{\dot{H}^1(A^\text{tr}\setminus A^d(2dr_0,\frac{1}{2}))} + \|\phi_d - \widetilde{\phi}_d\|^2_{\dot{H}^1(A^\text{tr}\setminus A^d(2dr_0,\frac{1}{2}))}.
    \end{equation*}
    These terms are respectively uniformly bounded by \cref{eq_4_th_Lapl}, an explicit computation (analogous to \cref{lem_estim_ext_lapl} on $A_\text{rem}$), and \cref{eq_3_th_Lapl}.
    This estimate reflects that, like $A^\text{tr}$, the set $A^d(2dr_0, \frac{1}{2})$ acts as a transition domain between the singular scale $d$ and the macroscopic domain.
    Consequently, the cross-term is bounded independently of $d$, establishing \cref{eq_5_th_Lapl}.

    Finally, to verify the uniform $L^2$-bound, we decompose
    \begin{equation*}
        \|\phi_d\|_{L^2(\Omega)} \leq\|\phi_d\|_{L^2(S^d)} + \|\widetilde{\phi}_d\|_{L^2(\F^d)} + \|\phi_d - \widetilde{\phi}_d\|_{L^2(\F^d)}.
    \end{equation*}
    The first term is $\bigo(1)$ since $\phi_d = 1$ on $S^d$.
    The second term is uniformly controlled by \cref{eq_def_phi_tilde_d}.
    For the third term, we note that $\phi_d - \widetilde{\phi}_d \in H^1(\F^d)$ and has a vanishing trace on the exterior boundary $\partial\Omega$ by \cref{eq_lapl} and \cref{eq:syst_local_phid}.
    The Poincaré inequality thus yields
    \begin{equation*}
        \|\phi_d - \widetilde{\phi}_d\|_{L^2(\F^d)} \leq C(\Omega) \|\phi_d - \widetilde{\phi}_d\|_{\dot{H}^1(\F^d)} = \bigo(1).
    \end{equation*}
    This establishes \cref{eq_6_th_Lapl} and completes the proof.
\end{proof}

\section{Stokes problem -- Proof of Theorem 2}\label{sec_stokes}

In this section, we extend the method developed for the Laplace problem to the Stokes equations, resulting in the proof of \cref{thm_stokes}.
As noted in the introduction, linearity allows us to restrict our attention to the canonical boundary velocities defined in \eqref{eq_cano_motions}.
We emphasize that this basis spans all planar rigid motions.
In particular, a rotation centered at an arbitrary point $x'=(x_1',x_2')\in\R^2$ is simply a linear combination of these canonical flows
\begin{equation}\label{eq_rot_orgin}
    u_{*,x'}^\circlearrowleft(x) = (x-x')^\perp = u_{*}^\circlearrowleft(x) - x_1' u_*^\perp(x) + x_2' u_*^\|(x).
\end{equation}
Consequently, establishing the asymptotics for the canonical data completely determines the flow for any general rigid inclusion.

Since the fluid domain $\F^d$ has a Lipschitz boundary, any divergence-free velocity field $u\in [H^1(\F^d)]^2$ satisfying the flux conditions $\int_{\partial\Omega}u\cdot n=0$ and $\int_{\partial S^d}u\cdot n=0$ admits a stream function $\psi\in H^2(\F^d)$, unique up to an additive constant, such that $u = \nabla^\perp \psi$. 
Let $\psi_d$ and $\psi_*$ denote the stream functions for $u_d$ and $u_*$, respectively.
On any connected boundary component, the identity $\nabla \psi_d = \nabla \psi_*$ is equivalent to matching the normal derivatives $\partial_n \psi_d = \partial_n \psi_*$ and traces differing by a constant $\psi_d = \psi_* + C$.
Hence, the no-slip condition $u_d=0$ on $\partial \Omega$ translates to $\partial_n\psi_d=0$ and $\psi_d = C_0$.
We uniquely determine $\psi_d$ by fixing $\psi_d = 0$ on $\partial \Omega$.
On $\partial S^d$, the condition $u_d=u_*$ becomes $\partial_n \psi_d = \partial_n \psi_*$ and $\psi_d = \psi_* + C_*$, where $C_* \in \R$ is an unknown constant that may depend on $d$.

The weak formulation of the Stokes problem ensures that $u_d\in [H^1(\F^d)]^2$ satisfies $\int_{\F^d} \nabla u_d : \nabla v = 0$ for all divergence-free test functions $v\in [H_0^1(\F^d)]^2$. 
Again, any such test function admits a representation $v = \nabla^\perp \phi$, where $\phi \in H^2(\F^d)$ satisfies $\phi = \partial_n\phi = 0$ on $\partial \Omega$, and $\phi = C$, $\partial_n\phi = 0$ on $\partial S^d$ for some arbitrary $C \in \R$.
Rewriting the weak formulation via stream functions gives $\int_{\F^d} \Delta \psi_d \Delta \phi = 0$.
Integrating by parts then yields
\begin{equation*}
    \int_{\F^d} (\Delta^2 \psi_d) \phi - C \int_{\partial S^d} \partial_n \Delta \psi_d = 0.
\end{equation*}
Choosing $\phi$ with compact support in $\F^d$ recovers the biharmonic equation $\Delta^2 \psi_d = 0$ in $\F^d$.
Conversely, choosing $C\neq0$ imposes the compatibility condition $\int_{\partial S^d} \partial_n \Delta \psi_d = 0$.
This constraint ensures that the pressure field $p_d$ is single-valued.
Indeed, letting $\tau = n^\perp$ denote the unit tangent vector to $\partial S^d$, we observe that $\partial_n \Delta \psi_d = \Delta u_d \cdot \tau = \nabla p_d \cdot \tau$.
The vanishing of this integral over the closed curve $\partial S^d$ is the necessary condition for the pressure $p_d$ to be globally well-defined.

Consequently, $\psi_d \in H^2(\F^d)$ is the solution to the biharmonic boundary value problem
\begin{equation}\label{eq_stream_stokes}
    \begin{cases}
        \Delta^2 \psi_d = 0 \quad &\text{in } \F^d,\\
        \psi_d = \psi_* + C_*, \quad \partial_n \psi_d = \partial_n \psi_* \quad&\text{on } \partial S^d,\\
        \psi_d = 0, \quad \partial_n \psi_d = 0 \quad &\text{on } \partial\Omega,\\
        \int_{\partial S^d} \partial_n \Delta \psi_d = 0,
    \end{cases}
\end{equation}
where $\psi_*$ belongs to the set of elementary boundary profiles
\begin{equation}\label{eq_elem_veloc}
    \psi^\perp_{*}(x)=x_1,\qquad \psi^\|_{*}(x)=-x_2,\quad\text{and}\quad \psi_{*}^\circlearrowleft(x)=\frac{1}{2}|x|^2.
\end{equation}
For a complete and rigorous treatment of this stream function formulation, we refer the reader to \cite[Section 5.2]{GiraultRaviart1986}. 
By linearity, we isolate the contribution of the constant $C_*$.
Let $\opsi_d$ denote the solution to the auxiliary problem
\begin{equation}\label{eq_stream_stokes_aux}
    \begin{cases}
       \Delta^2\opsi_d = 0 \quad &\text{in } \F^d,\\
       \opsi_d = \psi_*, \quad \partial_n\opsi_d = \partial_n \psi_* \quad&\text{on } \partial S^d,\\
       \opsi_d = 0, \quad \partial_n\opsi_d = 0 \quad &\text{on } \partial\Omega.
    \end{cases}
\end{equation}
We consider this problem with four elementary boundary profiles: the three profiles $\psi_*$ defined above, and the constant profile $\psi_{*}^c(x) := d$.
Standard arguments ensure that this problem admits a unique solution.
By superposition, the solution $\psi_d$ to \cref{eq_stream_stokes} decomposes as
\begin{equation}\label{eq_decomp_Cstar}
    \psi_d = \opsi_d + C_*d^{-1}\opsi_d^{c},
\end{equation}
where $\opsi_d$ and $\opsi_d^c$ denote the unique auxiliary solutions corresponding to an elementary profile from \cref{eq_elem_veloc} and the constant profile $\psi_*^c$, respectively.
Substituting this decomposition into the compatibility condition uniquely determines the constant $C_*$ (as studied in \cref{sec_return_init_stokes_pb_cst}), thereby establishing that $\psi_d$ is uniquely defined.
\begin{remark}
    By linearity, the normalization of the constant profile $\psi_{*}^c$ is arbitrary.
    This $O(d)$ scaling is chosen to match the local behavior of the linear profiles near the tip, simplifying subsequent asymptotic expansions.
\end{remark}

The proof of \cref{thm_stokes} proceeds in two steps.
First, in \cref{sec_order_opsi}, we adapt the method from the Laplace setting to construct explicit profiles $\widetilde{\psi}_d$ characterizing the leading-order behavior of $\opsi_d$.
\begin{proposition}[Singular profile for the auxiliary Stokes problem]\label{prop_stokes_aux}
    The stream functions $\widetilde{\psi}_d$, as constructed in \cref{thm_stokes}, satisfy
    \begin{subequations}
    \begin{empheq}[left=\empheqlbrace]{align}
        \|\opsi_d - \widetilde{\psi}_d\|^2_{\dot{H}^2(A^\text{tr})} &= \bigo(1),\label{eq_1_th_stokes_aux}\\
        \|\opsi_d \|^2_{\dot{H}^2(\Omega\setminus A^\text{tr})}&= \bigo(1).\label{eq_2_th_stokes_aux}
    \end{empheq}
    In particular, these estimates yield the asymptotic expansion
    \begin{equation}\label{eq_3_th_stokes_aux}
        \|\opsi_d\|^2_{\dot{H}^2(\Omega)} = \|\widetilde{\psi}_d\|^2_{\dot{H}^2(A^\text{tr})} + \bigo(1).
    \end{equation}
    Additionally, the potential remains bounded in $L^2$,
    \begin{equation}\label{eq_4_th_stokes_aux}
        \|\opsi_d\|^2_{\dot{H}^1(\Omega)}= \bigo(1).
    \end{equation}
    \end{subequations}
\end{proposition}
This result also holds for $\widetilde{\psi}_d^c$, the approximation of $\opsi_d^c$ associated with $\psi_{*}^c$.
The second step, carried out in \cref{sec_return_init_stokes_pb}, extends this result to the original Stokes problem.
Combining the decomposition \cref{eq_decomp_Cstar} with \cref{prop_stokes_aux}, we determine the asymptotic behavior of the constant $C_*$.
\begin{proposition}[Asymptotic behavior of $C_*$]\label{lem_constant_stokes}
    For any profile $\psi_*$ in \cref{eq_elem_veloc}, as $d\to 0$, the constant $C_*$ from \cref{eq_stream_stokes} satisfies
    \begin{equation*}
        C_* = \bigo(d).
    \end{equation*}
\end{proposition}
Together, these results allow us to determine the asymptotic behavior of the full stream function $\psi_d$, concluding the proof of \cref{thm_stokes}.
\subsection{Auxiliary problem -- Proof of Proposition 2}\label{sec_order_opsi}
\subsubsection{Construction of the Leading Profile}
Following the approach established in \cref{sec_const_prof_laplace} for the Laplace equation, we construct a local approximation to the auxiliary problem \cref{eq_stream_stokes_aux} on the annular region $A^d(\lambda dr_0, 2)$.
We begin by considering the near-tip localized problem under the rescaling $y=x/d$.
Under this change of variables, the rescaled boundary stream functions $\psi_{*,r}(y)=d^{-1}\psi_{*}(dy)$ are given by
\begin{equation*}
    \psi^\perp_{*,r}(y)=y_1,\qquad \psi^\|_{*,r}(y)=-y_2, \qquad \psi_{*,r}^\circlearrowleft(y)= \frac{d}{2}|y|^2,\quad \text{and}\quad \psi_{*,r}^c(y)= 1.
\end{equation*}
We first derive exact solutions on the left and right unbounded sectors.
\begin{lemma}[Solution in the unbounded sectors]\label{lem_expr_psiL_psiR}
    Let $\psi_L$ and $\psi_R$ be the functions that solve
    \begin{equation*}
        \begin{cases}
            \Delta^2 \psi_L =0\quad&\text{in }\C_L,\\
            \psi_L = \psi_{*,r}, \: \partial_n \psi_L = \partial_n \psi_{*,r}\: &\text{on }\{\theta = \pi-\theta_L\},\\
            \psi_L= \partial_n \psi_L=0\quad &\text{on }\{\theta = \pi\},
        \end{cases}\quad
        \begin{cases}
            \Delta^2 \psi_R =0\quad&\text{in }\C_R,\\
            \psi_R = \psi_{*,r}, \: \partial_n \psi_R = \partial_n \psi_{*,r}\: &\text{on }\{\theta = \theta_R\},\\
            \psi_R = \partial_n \psi_R =0 \quad &\text{on }\{\theta = 0\}.
        \end{cases} 
    \end{equation*}
    The exact solutions corresponding to the respective boundary data $\psi_{*,r}$ are
    \begin{subequations}
    \begin{itemize}
        \item For the perpendicular solid motion $\psi^\perp_{*,r}(y)=y_1$, we have $\psi_L^\perp(r,\theta) = rf_L^\perp(\theta)$, where
        \begin{equation}
            f_L^\perp(\theta) = B_L^\perp\big[-(\pi -\theta)\cos(\theta) - \sin(\theta)\big] - D_L^\perp(\pi - \theta)\sin(\theta),
        \end{equation}
        and $\psi_R^\perp(r,\theta) = rf_R^\perp(\theta)$, where
        \begin{equation}
            f_R^\perp(\theta) = B_R^\perp\big[\theta\cos(\theta) - \sin(\theta)\big] + D_R^\perp\theta\sin(\theta),
        \end{equation}
        with constants
        \begin{equation}
            \begin{aligned}
            &B_L^\perp= -\frac{\theta_L + \sin(\theta_L)\cos(\theta_L)}{\theta_L^2 - \sin^2(\theta_L)},\quad &&D_L^\perp= \frac{\sin^2(\theta_L)}{\theta_L^2 - \sin^2(\theta_L)},\\
            &B_R^\perp= \frac{\theta_R + \sin(\theta_R)\cos(\theta_R)}{\theta_R^2 - \sin^2(\theta_R)},\quad &&D_R^\perp= \frac{\sin^2(\theta_R)}{\theta_R^2 - \sin^2(\theta_R)}.
            \end{aligned}\label{eq:coeffs_perp}
        \end{equation}
        \item For the parallel solid motion $\psi^\|_{*,r}(y)=-y_2$, the solutions share the structural form of the perpendicular case, but with constants
        \begin{equation}
            \begin{aligned}
            &B_L^\|= -\frac{\sin^2(\theta_L)}{\theta_L^2 - \sin^2(\theta_L)},\quad &&D_L^\|= \frac{\theta_L - \sin(\theta_L)\cos(\theta_L)}{\theta_L^2 - \sin^2(\theta_L)},\\
            &B_R^\|= -\frac{\sin^2(\theta_R)}{\theta_R^2 - \sin^2(\theta_R)},\quad &&D_R^\|= -\frac{\theta_R - \sin(\theta_R)\cos(\theta_R)}{\theta_R^2 - \sin^2(\theta_R)}. 
            \end{aligned}\label{eq:coeffs_parallel}
        \end{equation}
        \item For the rotation solid motion $\psi_{*,r}^\circlearrowleft(y)=\frac{d}{2}|y|^2$, we have $\psi^\circlearrowleft_L(r,\theta) = \frac{d}{2} r^2 f^\circlearrowleft_L(\theta)$, where
        \begin{equation}
            f^\circlearrowleft_L(\theta) = C^\circlearrowleft_L\big(\cos(2\theta) - 1 \big) + D^\circlearrowleft_L\big(\sin(2\theta) + 2(\pi-\theta)\big),
        \end{equation}
        and $\psi^\circlearrowleft_R(r,\theta) =\frac{d}{2} r^2 f^\circlearrowleft_R(\theta)$, where
        \begin{equation}
            f^\circlearrowleft_R(\theta) = C_R^\circlearrowleft\big(\cos(2\theta) - 1 \big) + D_R^\circlearrowleft\big(\sin(2\theta) - 2\theta\big),
        \end{equation}
        with constants
        \begin{equation}
            \begin{aligned}
            &C_L^\circlearrowleft =\frac{-\sin(\theta_L)}{2(\sin(\theta_L)-\theta_L\cos(\theta_L))},\qquad &&D_L^\circlearrowleft=\frac{-\cos(\theta_L)}{2(\sin(\theta_L)-\theta_L\cos(\theta_L))},\\
            &C_R^\circlearrowleft =\frac{-\sin(\theta_R)}{2(\sin(\theta_R)-\theta_R\cos(\theta_R))},\qquad &&D_R^\circlearrowleft=\frac{\cos(\theta_R)}{2(\sin(\theta_R)-\theta_R\cos(\theta_R))}.
            \end{aligned}\label{eq:coeffs_rot}
        \end{equation}
        \item For the constant stream function $\psi_{*,r}^c(y) = 1$, the solutions are $\psi^c_L(r,\theta) = f^c_L(\theta)$ and $\psi^c_R(r,\theta) = f^c_R(\theta)$, where $f^c_L=f^\circlearrowleft_L$ and $f^c_R = f^\circlearrowleft_R$.
    \end{itemize}
    \end{subequations}
\end{lemma}
\begin{proof}
    We first derive the stream functions in the right sector $\C_R$ for the linear solid velocities $\psi_{*,r}^\perp=y_1$ and $\psi_{*,r}^\|=-y_2$.
    Postulating a separable solution of the form $\psi_R(r,\theta) = rf_R(\theta)$, the biharmonic equation in polar coordinates reduces to the ordinary differential equation
    \begin{equation}\label{eq_fct_ang_edo}
        \Delta^2 \psi_R = \frac{1}{r^3}\left(f_R''''(\theta) + 2 f_R''(\theta) + f_R(\theta)\right)=0.
    \end{equation}
    The general solution reads
    \begin{equation*}
        f_R(\theta) = (A_R+B_R\theta)\cos(\theta) + (C_R+D_R\theta)\sin(\theta).
    \end{equation*}
    The homogeneous boundary conditions $\psi_R=\partial_n\psi_R=0$ on $\{\theta=0\}$ impose $f_R(0) = f_R'(0) = 0$, which yields $A_R=0$ and $C_R=-B_R$.
    The profile therefore simplifies to
    \begin{equation*}
        f_R(\theta) = B_R\left[\theta\cos(\theta) - \sin(\theta)\right] + D_R\theta\sin(\theta).
    \end{equation*}
    The remaining constants are determined by the non-homogeneous boundary conditions $\psi_R = \psi_{*,r}$ and $\partial_n \psi_R = \partial_n \psi_{*,r}$ on $\{\theta = \theta_R\}$:
    \begin{itemize}
        \item For the perpendicular solid motion $\psi_{*,r}^\perp=r\cos(\theta)$, these conditions yield $f_R(\theta_R)=\cos(\theta_R)$ and $f'_R(\theta_R)=-\sin(\theta_R)$.
        Solving this linear system gives the coefficients $B_R^\perp$ and $D_R^\perp$.
        \item For the parallel solid motion $\psi_{*,r}^\|=-r\sin(\theta)$, the equivalent conditions $f_R(\theta_R)=-\sin(\theta_R)$ and $f'_R(\theta_R)=-\cos(\theta_R)$ analogously give $B_R^\|$ and $D_R^\|$.
    \end{itemize}
    By symmetry, applying the separable ansatz $\psi_L(r,\theta) = rf_L(\theta)$ in the left sector $\C_L$ and enforcing the homogeneous conditions at $\{\theta=\pi\}$ yields
    \begin{equation*}
        f_L(\theta) =B_L\big[-(\pi -\theta)\cos(\theta) - \sin(\theta)\big] - D_L(\pi - \theta)\sin(\theta).
    \end{equation*}
    The boundary conditions on $\{\theta=\pi -\theta_L\}$ reduce to
    \begin{itemize}
        \item For the perpendicular solid motion $\psi_{*,r}^\perp$, as $f_L(\pi-\theta_L)=\cos(\pi-\theta_L)$ and $f'_L(\pi-\theta_L)=-\sin(\pi-\theta_L)$ which determine $B_L^\perp$ and $D_L^\perp$.
        \item For the parallel solid motion $\psi_{*,r}^\|$, as $f_L(\pi-\theta_L)=-\sin(\pi-\theta_L)$ and $f'_L(\pi-\theta_L)=-\cos(\pi-\theta_L)$ yielding $B_L^\|$ and $D_L^\|$.
    \end{itemize}

    Next, we address the rotational solid motion $\psi_{*,r}^\circlearrowleft = \frac{d}{2}r^2$.
    We postulate an ansatz of the form $\psi_R^\circlearrowleft(r,\theta) = \frac{d}{2}r^2f_R^\circlearrowleft(\theta)$.
    The biharmonic equation reduces to 
    \begin{equation}\label{eq_fct_ang_edo_rot}
        f_R^\circlearrowleft{}''''(\theta) + 4 f_R^\circlearrowleft{}''(\theta)=0,
    \end{equation}
    yielding the general solution
    \begin{equation*}
        f_R(\theta) = A_R+B_R\theta + C_R\cos(2\theta) + D_R\sin(2\theta).
    \end{equation*}
    The homogeneous conditions $f_R(0) = f_R'(0) = 0$ at $\{\theta=0\}$ reduce the profile to
    \begin{equation*}
        f_R(\theta) = C_R\big(\cos(2\theta) - 1 \big) + D_R\big(\sin(2\theta) - 2\theta\big).
    \end{equation*}
    The boundary conditions at $\{\theta=\theta_R\}$ correspond to $f_R(\theta_R)=1$ and $f'_R(\theta_R)=0$, which resolve to $C_R^\circlearrowleft$ and $D_R^\circlearrowleft$.
    Analogously, for the left sector, setting the biharmonic profile $\psi^\circlearrowleft_L(r,\theta) = \frac{d}{2}r^2f_L(\theta)$ and applying the conditions on $\{\theta=\pi\}$ leads to
    \begin{equation*}
        f_L(\theta) = C_L\big(\cos(2\theta) - 1 \big) + D_L\big(\sin(2\theta) + 2(\pi-\theta)\big).
    \end{equation*}
    Matching the boundary conditions at $\{\theta=\pi - \theta_L\}$, which are $f_L(\pi-\theta_L)=1$ and $f'_L(\pi-\theta_L)=0$ yields $C_L^\circlearrowleft$ and $D_L^\circlearrowleft$.

    Finally, for the constant stream function $\psi_{*,r}^c(y)= 1$, we postulate solutions $\psi_R^c(r,\theta) = f_R(\theta)$ and $\psi_L^c(r,\theta) = f_L(\theta)$.
    The absence of radial dependence reduces the biharmonic equation to $f'''' + 4f''=0$.
    Because the boundary conditions at the domain walls are identical to those of the rotational case, it follows immediately that $f_R^c=f_R^\circlearrowleft$ and $f_L^c=f_L^\circlearrowleft$, completing the proof.
\end{proof}
\begin{remark}\label{rk_bound_data}
    To satisfy the boundary conditions on $d^{-1}S^d$, the canonical solutions $\psi_L$ and $\psi_R$ must be translated by $y_L$ and $y_R$, respectively.
    However, because the boundary data $\psi_{*,r}$ depend explicitly on the spatial coordinates, a translation yields the correct profiles only for the constant mode.
    The parallel mode is also unaffected because the horizontal coordinates of $y_L$ and $y_R$ vanish, leading $\psi^\|_{*,r}(y) =-y_2 = \psi^\|_{*,r}(y-y_{L/R})$.
    For the perpendicular and rotational modes, a naive translation induces lower-order polynomial terms.
    For instance, $\psi^\perp_{*,r}(y-y_R) = y_1 - (y_R\cdot e_1) = \psi^\perp_{*,r}(y) - (y_R\cdot e_1)\psi^c_{*,r}(y)$.
    By linearity, the correct boundary-matched biharmonic profiles are therefore given by
    \begin{equation*}
        \begin{cases}
        \widetilde{\psi}_R^\perp(y) = \psi^\perp_R(y- y_R) + (y_R\cdot e_1) \psi^c_R(y- y_R),\\
        \widetilde{\psi}_R^\circlearrowleft(y) = \psi^\circlearrowleft_R(y- y_R) + d(y_R\cdot e_1) \psi^\perp_R(y- y_R) + \frac{d}{2}|y_R|^2 \psi^c_R(y- y_R),
        \end{cases}
    \end{equation*}
    with analogous expressions for the left profiles $\widetilde{\psi}_L$.
    As shown below, these correction terms are subdominant and may be neglected in the leading-order analysis.
    They are needed only in the final error estimates, where the precise boundary traces of the approximation become critical.
\end{remark}
We define the piecewise local approximation profile as
\begin{equation*}
    \widetilde{\psi}_r(y) :=
    \begin{cases}
        \psi_L(y-y_L)\quad &\text{in } A^1_L(\lambda r_0,2d^{-1}),\\
        \psi_R(y-y_R)\quad &\text{in } A^1_R(\lambda r_0,2d^{-1}).
    \end{cases}
\end{equation*}
A direct analogue of the proof of \cref{lem_jonct_lapl} yields the following result.
\begin{lemma}[Local profile]\label{lem_H2_regularity_combined_function}
    The function $\widetilde{\psi}_r$ is well-defined, smooth on $A^1(\lambda r_0,2d^{-1})$, and satisfies
    \begin{equation*}
        \begin{cases}
            \Delta^2 \widetilde{\psi}_r = 0\quad & \text{in }  A^1(\lambda r_0,2d^{-1}), \\
            \widetilde{\psi}_r = \partial_n \widetilde{\psi}_r = 0\quad & \text{on } \partial [d^{-1}\Omega] \cap \partial A^1(\lambda r_0,2d^{-1}).
        \end{cases}
    \end{equation*}
\end{lemma}
Rescaling back to the original geometry, we define the gap profile on $A^d(\lambda dr_0,2)$ by
\begin{equation}\label{eq_def_psi_gap}
    \widetilde{\psi}_\text{gap}(x) := d\widetilde{\psi}_r\left(\frac{x}{d}\right) =\begin{cases}
            d\psi_L\big(d^{-1}(x-x_L)\big)\quad &\text{in } A^d_L(\lambda d r_0,2),\\[0.5em]
            d\psi_R\big(d^{-1}(x-x_R)\big)\quad &\text{in } A^d_R(\lambda d r_0,2).
        \end{cases}
\end{equation}
For linear solid motions, the gap profile takes the form
\begin{equation*}
    \widetilde{\psi}_\text{gap}(x) = \begin{cases}
            |x-x_L|f_L\big(\arg(x-x_L)\big)\quad &\text{in } A^d_L(\lambda d r_0,2),\\[0.5em]
            |x-x_R|f_R\big(\arg(x-x_R)\big)\quad &\text{in } A^d_R(\lambda d r_0,2).
        \end{cases}
\end{equation*}
Consequently, for $x \in A^d_{L/R}(\lambda d r_0,2)$ such that $0 < \mu_1 < |x-x_{L/R}| < \mu_2$, this profile satisfies the bounds
\begin{subequations}
\begin{equation}\label{eq_bound_profil_gap_psi}
    |\widetilde{\psi}_\text{gap}| \leq \mu_2 C(\theta_L,\theta_R), \quad |\nabla\widetilde{\psi}_\text{gap}| \leq C(\theta_L,\theta_R),\quad |\nabla^2\widetilde{\psi}_\text{gap}| \leq\mu_1^{-1} C(\theta_L,\theta_R).
\end{equation}
Similarly, for rotational solid motions, the profile reads
\begin{equation*}
    \widetilde{\psi}_\text{gap}^\circlearrowleft(x) = \begin{cases}
            \frac{1}{2}|x-x_L|^2f_L^\circlearrowleft\big(\arg(x-x_L)\big)\quad &\text{in } A^d_L(\lambda d r_0,2),\\[0.5em]
            \frac{1}{2}|x-x_R|^2f_R^\circlearrowleft\big(\arg(x-x_R)\big)\quad &\text{in } A^d_R(\lambda d r_0,2).
        \end{cases}
\end{equation*}
For $|x-x_{L/R}| < \mu_2$, this yields the corresponding estimates
\begin{equation}\label{eq_bound_profil_gap_psirot}
    |\widetilde{\psi}_\text{gap}^\circlearrowleft| \leq \mu_2^2 C(\theta_L,\theta_R), \quad |\nabla\widetilde{\psi}_\text{gap}^\circlearrowleft| \leq \mu_2 C(\theta_L,\theta_R),\quad |\nabla^2\widetilde{\psi}_\text{gap}^\circlearrowleft| \leq C(\theta_L,\theta_R).
\end{equation}
Finally, for the constant boundary condition $\psi_*^c = d$, we have
\begin{equation*}
    \widetilde{\psi}_\text{gap}^c(x) = \begin{cases}
            d^{} f_L^c\big(\arg(x-x_L)\big)\quad &\text{in } A^d_L(\lambda d r_0,2),\\[0.5em]
            d^{} f_R^c\big(\arg(x-x_R)\big)\quad &\text{in } A^d_R(\lambda d r_0,2),
        \end{cases}
\end{equation*}
which implies that for $|x-x_{L/R}| > \mu_1$,
\begin{equation}\label{eq_bound_profil_gap_psiconst}
    |\widetilde{\psi}_\text{gap}^c| \leq d^{} C(\theta_L,\theta_R), \quad |\nabla\widetilde{\psi}_\text{gap}^c| \leq d^{} \mu_1^{-1} C(\theta_L,\theta_R),\quad |\nabla^2\widetilde{\psi}_\text{gap}^c| \leq d^{} \mu_1^{-2}C(\theta_L,\theta_R).
\end{equation}
\end{subequations}

Furthermore, \cref{lem_H2_regularity_combined_function} ensures that $\widetilde{\psi}_\text{gap}$ satisfies
\begin{equation}\label{eq_syst_local_psigap}
    \begin{cases}
        \Delta^2 \widetilde{\psi}_\text{gap} = 0 \quad& \text{in }A^d(\lambda dr_0,2), \\
        \widetilde{\psi}_\text{gap} = \partial_n \widetilde{\psi}_\text{gap} = 0 \quad& \text{on } \partial \Omega \cap \partial [A^d(\lambda dr_0,2)].
    \end{cases}
\end{equation}
\begin{lemma}[Local profile energy]\label{lem_estim_tildepsir_bilapl}
    As $d \to 0$, the constructed profiles $\widetilde{\psi}_\text{gap}$ satisfy
    \begin{equation*}
        \begin{cases}
        \|\widetilde{\psi}_\text{gap}^\perp\|_{\dot{H}^2(A^\text{tr})}^2 &= 2\left[\frac{\theta_L + \sin(\theta_L)\cos(\theta_L)}{\theta_L^2 - \sin^2(\theta_L)} + \frac{\theta_R + \sin(\theta_R)\cos(\theta_R)}{\theta_R^2 - \sin^2(\theta_R)}\right]|\log(d)| + \bigo(1),\\[0.5em]
        \|\widetilde{\psi}_\text{gap}^\|\|_{\dot{H}^2(A^\text{tr})}^2 &= 2\left[\frac{\theta_L - \sin(\theta_L)\cos(\theta_L)}{\theta_L^2 - \sin^2(\theta_L)} + \frac{\theta_R - \sin(\theta_R)\cos(\theta_R)}{\theta_R^2 - \sin^2(\theta_R)}\right]|\log(d)| + \bigo(1),\\[0.5em]
        \|\widetilde{\psi}_\text{gap}^\circlearrowleft \|_{\dot{H}^2(A^\text{tr})}^2 &= \bigo(1),\\[0.5em]
        \|\widetilde{\psi}_\text{gap}^c\|_{\dot{H}^2(A^\text{tr})}^2 &= \bigo(1).
        \end{cases}
    \end{equation*}
\end{lemma}
\begin{proof}
    The proof parallels that of \cref{lem_estim_tildepsir_lapl}, with more involved computations.
    We decompose the energy as
    \begin{equation*}
        \|\nabla^2\widetilde{\psi}_\text{gap}\|_{L^2(A^\text{tr})}^2 = \int_{A^\text{tr}_L-x_L}d^{-2}|\nabla^2\psi_L(d^{-1}x)|^2 + \int_{A^\text{tr}_R-x_R}d^{-2}|\nabla^2\psi_R(d^{-1}x)|^2.
    \end{equation*}
    Focusing on $A^\text{tr}_R -x_R$, we partition the domain into a principal region $A_\text{main} := \{(r,\theta) : r\in(dr_1,1),~\theta\in(0,\theta_R)\}$, for a sufficiently large radius $r_1>0$ depending only on $\theta_L$ and $\theta_R$, and a remainder $A_\text{rem}$.
    In polar coordinates, the Hessian reads
    \begin{equation*}
        \nabla^2\psi = \partial_{rr}^2 \psi \,e_r\otimes e_r +\frac{1}{r}\left(\partial_{r\theta}^2 \psi - \frac{1}{r}\partial_\theta \psi\right) \,(e_r\otimes e_\theta + e_\theta \otimes e_r) + \frac{1}{r^2}( r \partial_r \psi + \partial_{\theta\theta}^2 \psi ) \,e_\theta\otimes e_\theta.
    \end{equation*}
    For the linear solid motions $\psi^{\perp}_{*}$ and $\psi^{\|}_{*}$, \cref{lem_expr_psiL_psiR} ensures that $\psi_R(r,\theta) = r f_R(\theta)$.
    Evaluating the $L^2$-norm of the Hessian over $A_\text{main}$ yields
    \begin{equation*}
        \int_{A_\text{main}}d^{-2}|\nabla^2\psi_R(d^{-1}x)|^2 = \int^{\theta_R}_{0}\int_{dr_1}^{1}\frac{\big(f_R(\theta) + f_R''(\theta)\big)^2}{r^2}rdrd\theta = K_R(\theta_R)|\log(d)|+\bigo(1).
    \end{equation*}
    According to the explicit profile $f_R(\theta) = B_R[\theta\cos(\theta) - \sin(\theta)] + D_R\theta\sin(\theta)$, a direct computation gives
    \begin{align*}
        K_R(\theta_R) = 2B_R^2\big(\theta_R-\sin(\theta_R)\cos(\theta_R)\big)- 4 B_RD_R\sin^2(\theta_R) + 2D_R^2\big(\theta_R+\sin(\theta_R)\cos(\theta_R)\big).
    \end{align*}
    Using the explicit expression of the parameters from \cref{eq:coeffs_perp} and \cref{eq:coeffs_parallel}, we extract the leading-order coefficients
    \begin{equation*}
        K_R^\perp(\theta_R) = 2 \frac{\theta_R + \sin(\theta_R)\cos(\theta_R)}{\theta_R^2 - \sin^2(\theta_R)} = 2B_R^\perp,\quad K_R^\|(\theta_R)= 2 \frac{\theta_R - \sin(\theta_R)\cos(\theta_R)}{\theta_R^2 - \sin^2(\theta_R)} = -2D_R^\|.
    \end{equation*}
    Since the contact angle $\theta_R$ is strictly positive, these quantities are well-defined.
    The remainder $A_\text{rem}$ consists of an inner component, contained within an annulus of radii proportional to $d$, and an outer component, contained within an annulus with $d$-independent radii.
    An analogous calculation shows its contribution is uniformly controlled.
    Hence,
    \begin{equation*}
        \|\nabla^2\widetilde{\psi}_\text{gap}\|_{L^2(A^\text{tr}_R)}^2 = K_R(\theta_R)|\log(d)|+\bigo(1).
    \end{equation*}

    A fully analogous argument applies to the left domain.
    Integrating over the corresponding angular sector $(\pi-\theta_L,\pi)$ yields
    \begin{equation*}
        \|\nabla^2\widetilde{\psi}_\text{gap}\|_{L^2(A^\text{tr}_L)}^2 = K_L(\theta_L)|\log(d)|+\bigo(1),
    \end{equation*}
    where the geometric constant takes the form
    \begin{align*}
        K_L(\theta_L) = 2B_L^2\big(\theta_L-\sin(\theta_L)\cos(\theta_L)\big)+ 4 B_LD_L\sin^2(\theta_L) + 2D_L^2\big(\theta_L+\sin(\theta_L)\cos(\theta_L)\big).
    \end{align*}
    Substituting the explicit definitions for $B_L$ and $D_L$ gives
    \begin{equation*}
        K_L^\perp(\theta_L) = 2 \frac{\theta_L + \sin(\theta_L)\cos(\theta_L)}{\theta_L^2 - \sin^2(\theta_L)} = -2B_L^\perp,\quad K_L^\|(\theta_L)= 2 \frac{\theta_L - \sin(\theta_L)\cos(\theta_L)}{\theta_L^2 - \sin^2(\theta_L)} = 2D_L^\|.
    \end{equation*}
    Combining these left and right estimates establishes the first two assertions of the lemma.

    We now turn to the rotational case, characterized by the scaled profile $\psi_R^\circlearrowleft(r,\theta) = \frac{d}{2}r^2f_R^\circlearrowleft(\theta)$.
    First, its Hessian is independent of $r$ since
    \begin{equation*}
        \nabla^2\psi_R^\circlearrowleft = d\left[f_R^\circlearrowleft e_r\otimes e_r + \frac{1}{2}f_R^\circlearrowleft\!\phantom{.}'\big(e_r\otimes e_\theta + e_\theta \otimes e_r\big) + \big(f_R^\circlearrowleft + \frac{1}{2}f_R^\circlearrowleft\!\phantom{.}''\big)e_\theta\otimes e_\theta\right].
    \end{equation*}
    Integration over $A_\text{main}$ then gives
    \begin{equation*}
        \int_{A_\text{main}}d^{-2}|\nabla^2\psi_R^\circlearrowleft(d^{-1}x)|^2 = K_R^\circlearrowleft(\theta_R)\int_{dr_1}^{1}rdr = \frac{1}{2} K_R^\circlearrowleft(\theta_R) + \bigo(d^2).
    \end{equation*}
    As $A_\text{rem}$ is confined to a left component whose radii scale with $d$ and a macroscopic right component with fixed radii, we obtain
    \begin{equation*}
        \|\nabla^2\widetilde{\psi}_\text{gap}^\circlearrowleft\|_{L^2(A^\text{tr}_R)}^2 = \frac{1}{2} K_R^\circlearrowleft(\theta_R) + \bigo(d^2),
    \end{equation*}
    An identical procedure for the left domain yields an analogous $\bigo(1)$ bound.
    Summing these terms finally gives
    \begin{equation*}
       \|\widetilde{\psi}_\text{gap}^\circlearrowleft\|_{\dot{H}^2(A^\text{tr})}^2 =\bigo(1).
    \end{equation*}
    We omit the exact expression for this geometric constant.
    The background energy outside the gap is $\bigo(1)$, the precise value of this constant offers no further information into the global system.

    Finally, for the constant boundary condition $\psi_{*}^c$, the profile behaves as $\psi_R^c(r,\theta) = f_R^c(\theta)$.
    Integrating the corresponding Hessian over the principal region $A_\text{main}$ yields
    \begin{equation*}
        \int_{A_\text{main}}d^{-2}|\nabla^2\psi_R^c(d^{-1}x)|^2 = d^2K_R^c(\theta_R)\int_{dr_1}^{1}r^{-3}dr = \frac{K_R^c(\theta_R)}{2r_1^2} + \bigo(d^2).
    \end{equation*}
    As before, the contribution of the remainder region $A_\text{rem}$ is uniformly bounded.
    The same argument applies to the left domain, yielding
    \begin{equation*}
        \|\widetilde{\psi}_\text{gap}^c\|_{\dot{H}^2(A^\text{tr})}^2 = \bigo(1).
    \end{equation*}
\end{proof}
\subsubsection{Error estimation}
To apply the stability argument, we must extend the local profile to the entire domain $\F^d$.
We employ the same lifting strategy used for $\widetilde{\phi}_d$.
Using the cutoffs from \cref{eq_def_chi_ext}, we define
\begin{equation}\label{eq_def_psi_tilde_d}
    \widetilde{\psi}_d(x)=\left[1-\chi_1\left(\frac{|x|}{dr_0}\right)\right]\chi_2(|x|)\widetilde{\psi}_\text{gap}(x).
\end{equation}
By construction, $\widetilde{\psi}_d$ is supported in $A^d(\lambda dr_0,2)$ and coincides with $\widetilde{\psi}_\text{gap}$ on the transition annulus $A^\text{tr} = A^d(dr_0,1)$.
This global approximation satisfies
\begin{equation}\label{eq_syst_local_psid}
    \begin{cases}
        \Delta^2 \widetilde{\psi}_d = 0 \quad& \text{in } A^\text{tr}, \\
        \widetilde{\psi}_d = \partial_n  \widetilde{\psi}_d = 0 \quad& \text{on } \partial \Omega.
    \end{cases}
\end{equation}
\begin{lemma}[Lifting energy]\label{lem_estim_ext_psid}
    As $d\to 0$, the energy of the extension $\widetilde{\psi}_d$ outside $A^\text{tr}$ satisfies
    \begin{equation*}
        \|\widetilde{\psi}_d\|_{\dot{H}^2(\F^d\setminus A^\text{tr})} = \bigo(1).
    \end{equation*}
\end{lemma}
\begin{proof}
    As done for \cref{lem_estim_ext_lapl}, decompose the integration domain $\supp(\widetilde{\psi}_d) \setminus A^\text{tr}$ into the outer annulus $ A^\text{out}=A^d(1,2)$ and the inner annulus $A^\text{in}= A^d(\lambda d r_0,dr_0)$.
    In $A^\text{out}$, the inner cutoff satisfies $\chi_1 = 0$.
    Therefore, the $\dot{H}^2$-contribution is bounded by $C(\theta_L, \theta_R)$ due to the uniform $L^\infty$-bounds on $\widetilde{\psi}_\text{gap}$ and $\chi_2$, along with their derivatives over $A^\text{out}$, which is a domain of measure $\bigo(1)$.
    In the inner region $A^\text{in}$, the outer cutoff satisfies $\chi_2 = 1$.
    Expanding the Hessian gives
    \begin{equation*}
        \big|\nabla^2 \widetilde{\psi}_d(x)\big| \leq C\left[\frac{1}{d^2r_0^2}\big|\widetilde{\psi}_\text{gap}(x)\big| + \frac{1}{dr_0}\big|\nabla\widetilde{\psi}_\text{gap}(x)\big|+ \big|\nabla^2\widetilde{\psi}_\text{gap}(x)\big|\right].
    \end{equation*}
    Within $A^\text{in}$, the radial coordinate satisfies $|x| \sim d$.
    For linear solid motions and constant boundary conditions, \cref{eq_bound_profil_gap_psi} and \cref{eq_bound_profil_gap_psiconst} ensure that $|\widetilde{\psi}_\text{gap}| = \bigo(d)$, $|\nabla \widetilde{\psi}_\text{gap}| = \bigo(1)$, and $|\nabla^2 \widetilde{\psi}_\text{gap}| = \bigo(d^{-1})$.
    For rotational motions, \cref{eq_bound_profil_gap_psirot} yields even stronger scaling.
    Consequently, $\big|\nabla^2 \widetilde{\psi}_d\big| = \bigo(d^{-1})$ throughout $A^\text{in}$.
    Since $|A^\text{in}| = \bigo(d^2)$, integrating this pointwise bound completes the proof.
\end{proof}
\begin{lemma}[Approximation error]\label{lem_estim_error_approx_bilapl}
    As $d\to0$, the approximation error satisfies
    \begin{equation*}
        \|\opsi_d-\widetilde{\psi}_d\|_{\dot{H}^2(\F^d)} = \bigo(1).
    \end{equation*}
\end{lemma}
\begin{proof}
    We adapt the strategy of \cref{lem_estim_error_approx_lapl}.
    This approach requires that the difference $\opsi_d - \widetilde{\psi}_d$ vanishes on the boundary of the transition domain $\partial\F^d\cap \partial A^\text{tr}$, a condition that fails as noted in \cref{rk_bound_data}.
    To resolve this, we define a modified approximation $\widetilde{\psi}^\dagger_d$ on $A_R^d(\lambda dr_0,2)$ by
    \begin{equation*}
        \widetilde{\psi}^\dagger_d =
        \begin{cases}
            \widetilde{\psi}_d&\quad \text{for } \psi_*^c,\psi_*^\|,\\[0.5em]
            \widetilde{\psi}_d^\perp +\frac{1}{d}(x_R\cdot e_1)\widetilde{\psi}_d^c&\quad \text{for } \psi_*^\perp, \\[0.5em]
            \widetilde{\psi}_d^\circlearrowleft +(x_R\cdot e_1)\widetilde{\psi}_d^\perp + \frac{1}{2d}|x_R|^2\widetilde{\psi}_d^c&\quad \text{for } \psi_*^\circlearrowleft,
        \end{cases}
    \end{equation*}
    with symmetric definitions on $A_L^d(\lambda dr_0,2)$.
    This modification is small in the energy norm.
    Indeed, \cref{lem_estim_tildepsir_bilapl,lem_estim_ext_psid} yield $\|\widetilde{\psi}_d^c\|_{\dot{H}^2(\F^d)}^2 = \bigo(1)$ and $\|\widetilde{\psi}_d^\perp\|_{\dot{H}^2(\F^d)}^2 = \bigo(|\log(d)|)$.
    Since $x_R \sim d e_1$, it follows that 
    \begin{equation}\label{eq_diff_gag}
        \|\widetilde{\psi}^\dagger_d - \widetilde{\psi}_d \|_{\dot{H}^2(\F^d)} = \bigo(1).
    \end{equation}
    
    Define the error $\psi^\text{err} = \opsi_d - \widetilde{\psi}^\dagger_d$.
    By construction, this field vanishes on $\partial\F^d\cap \partial A^\text{tr}$, allowing us to carry out the arguments from \cref{lem_estim_error_approx_lapl}.
    The weak formulation of $\opsi_d$ yields
    \begin{equation}\label{eq_var_err}
        \int_{\F^d} \Delta \psi^\text{err} \Delta \phi = -\int_{\F^d} \Delta \widetilde{\psi}^\dagger_d \Delta \phi,
    \end{equation}
    for all test functions $\phi$ in the admissible space $V = \{ \phi \in H^2(\F^d) : \phi = \partial_n\phi = 0 \text{ on } \partial \Omega, \text{ and } \phi = C, \, \partial_n\phi = 0 \text{ on } \partial S^d \text{ for some } C \in \R \}$.
    Anticipating its explicit construction, we postulate the existence of a boundary lifting $\psi_\text{BC} \in H^2(\F^d)$ such that $\psi^\text{err}_0 := \psi^\text{err} - \psi_\text{BC} \in H^2_0(\F^d)$.
    Restricting the test space $V$ to $H^2_0(\F^d)$, we obtain
    \begin{equation*}
        \int_{\F^d} \Delta \psi^\text{err}_0 \Delta \phi = -\int_{\F^d} \Delta \widetilde{\psi}^\dagger_d \Delta \phi - \int_{\F^d} \Delta \psi_\text{BC} \Delta \phi, \qquad \forall \phi \in H^2_0(\F^d).
    \end{equation*}
    Choosing $\phi = \psi^\text{err}_0$, it remains to show that the right-hand side defines a continuous linear functional on $H^2_0(\F^d)$ with a norm uniformly bounded in $d$.

    We first bound the term involving $\widetilde{\psi}_d^\dagger$.
    Integration by parts yields
    \begin{equation*}
        \int_{\F^d} \Delta \widetilde{\psi}_d^\dagger \Delta \phi = \int_{\F^d} \Delta^2 \widetilde{\psi}^\dagger_d \phi, \qquad \forall \phi \in H^2_0(\F^d).
    \end{equation*}
    By \cref{eq_syst_local_psid}, $\Delta^2 \widetilde{\psi}^\dagger_d$ vanishes identically on the transition annulus $A^\text{tr}$, restricting its support to $A^\text{in} \cup A^\text{out}$.
    In the inner region $A^\text{in}=A^d(\lambda dr_0,dr_0)$, each derivative acting on the cutoff $\chi_1$ or the profile $\widetilde{\psi}_\text{gap}$ introduces a scaling factor of $d^{-1}$.
    Since $\widetilde{\psi}^\dagger_d$ initially scales at least as $d$ in this domain, it follows that $|\Delta^2 \widetilde{\psi}^\dagger_d| = \bigo(d^{-3})$.
    Given $|A^\text{in}| = \bigo(d^2)$, integrating this pointwise bound yields $\| \Delta^2 \widetilde{\psi}^\dagger_d \|_{L^2(A^\text{in})} = \bigo(d^{-2})$.
    Moreover, since $\phi \in H^2_0(\F^d)$ satisfies $\phi = \partial_n \phi = 0$ on $\partial \Omega \cap \partial A^\text{in}$ and $\diam(A^\text{in}) = \bigo(d)$, a double application of the Poincaré inequality on $A^\text{in}$ gives $\| \phi \|_{L^2(A^\text{in})} \leq C d^2 \| \phi \|_{\dot{H}^2(A^\text{in})}$.
    Consequently,
    \begin{equation*}
        \left|\int_{A^\text{in}} \Delta^2 \widetilde{\psi}^\dagger_d \phi\right| \leq \| \Delta^2 \widetilde{\psi}^\dagger_d \|_{L^2(A^\text{in})} \| \phi \|_{L^2(A^\text{in})} \leq C \| \phi \|_{\dot{H}^2(A^\text{in})}.
    \end{equation*}
    In the far-field region $A^\text{out}=A^d(1,2)$, the remaining cutoff in the definition of $\widetilde{\psi}^\dagger_d$ is independent of $d$.
    As $\widetilde{\psi}_\text{gap}$ and its derivatives are uniformly bounded, we obtain $\|\Delta^2 \widetilde{\psi}^\dagger_d \|_{L^2(A^\text{out})} \leq C$.
    Combining these estimates yields
    \begin{equation}\label{eq_cont_psi_d_H20}
        \left|\int_{\F^d} \Delta^2 \widetilde{\psi}^\dagger_d \phi\right| \leq C \| \phi\|_{\dot{H}^2(\F^d)}, \qquad \forall \phi \in H^2_0(\F^d).
    \end{equation}

    Next, we establish the existence of a boundary lifting $\psi_\text{BC}$ satisfying
    \begin{equation}\label{eq:bound_lifting_psi}
        \|\psi_\text{BC}\|_{\dot{H}^2(\F^d)} = \bigo(1).
    \end{equation}
    By construction, $\widetilde{\psi}^\dagger_d=\psi_*$ and $\partial_n\widetilde{\psi}^\dagger_d=\partial_n\psi_*$ on $\partial \F^d \cap \partial A^\text{tr}$, 
    while the boundary where this condition fails decouples into an inner tip contribution ($\partial \F^d \cap \partial G^d(dr_0)$) and an outer far-field contribution ($\partial \F^d \setminus \partial G^d(1)$).
    Accordingly, we split the lifting as $\psi_\text{BC} = \psi_\text{BC}^\text{tip} + \psi_\text{BC}^\text{far}$.

    The inner lifting $\psi_\text{BC}^\text{tip}$ must satisfy $\psi_\text{BC}^\text{tip} = \psi_*-\widetilde{\psi}^\dagger_d$ and $\partial_n \psi_\text{BC}^\text{tip} = \partial_n\psi_*- \partial_n\widetilde{\psi}^\dagger_d$ on $\partial S^d \cap \partial G^d(dr_0)$ and vanish on the remainder of $\partial\F^d$.
    We set
    \begin{equation*}
        \psi_\text{BC}^\text{tip}(x) := \chi_\text{tip}\left(\frac{\dist(x,S^d)}{d/4}\right)\chi_\text{tip}\left(\frac{|x|}{dr_0}\right)[\psi_*(x)-\widetilde{\psi}^\dagger_d(x)],
    \end{equation*}
    where the cutoff $\chi_\text{tip} \in C_c^\infty(\R_+)$ satisfies $\mathds{1}_{[0,1)} \leq \chi_\text{tip} \leq \mathds{1}_{[0,2)}$.
    The normal cutoff localizes the lifting to a $d/4$-tubular neighborhood of $\partial S^d$, enforcing $\psi_\text{BC}^\text{tip} = \psi_* - \widetilde{\psi}^\dagger_d$ and $\partial_n\psi_\text{BC}^\text{tip} =\partial_n(\psi_* - \widetilde{\psi}^\dagger_d)$ on $\partial S^d \cap \partial G^d(dr_0)$ while vanishing at the bottom boundary $\partial\Omega$.
    Simultaneously, the radial cutoff restricts the support to $G^d(2dr_0)$.
    On the intermediate top boundary $\partial S^d \cap \partial A^d(dr_0,2dr_0)$, the construction of $\widetilde{\psi}^\dagger_d$ guarantees $\widetilde{\psi}^\dagger_d = \psi_{*}$ and $\partial_n \widetilde{\psi}^\dagger_d = \partial_n\psi_{*}$, meaning that $\psi_\text{BC}^\text{tip}=\partial_n\psi_\text{BC}^\text{tip}=0$ on this portion.
    It follows that $\psi_\text{BC}^\text{tip}$ correctly lifts the data.
    To bound the energy, we observe that each derivative introduces a scaling factor of at most $d^{-1}$.
    Indeed, using $\supp(\widetilde{\psi}^\dagger_d)\subset A^d(\lambda d r_0,2)$ and the bounds \cref{eq_bound_profil_gap_psi}--\cref{eq_bound_profil_gap_psiconst}, we control the contribution of $\widetilde{\psi}_\text{gap}$.
    Since $\psi_* - \widetilde{\psi}^\dagger_d$ initially scales as $\bigo(d)$, it follows that $|\nabla^2\psi_\text{BC}^\text{tip}|= \bigo(d^{-1})$.
    This is balanced by the measure of the radial support $G^d(2dr_0)$, which is $\bigo(d^2)$.
    Thus, $\|\psi_\text{BC}^\text{tip}\|_{\dot{H}^2(\F^d)} = \bigo(1)$.

    We define the outer lifting on $\F^d$ by 
    \begin{equation*}
        \psi_\text{BC}^\text{far}(x) := \chi_\text{tip}\big(\dist(x,S^d)\big)\big(1-\chi_\text{tip}(2|x|)\big)\big[\psi_*(x)-\widetilde{\psi}^\dagger_d(x)\big].
    \end{equation*}
    This function satisfies $\psi_\text{BC}^\text{far} = \psi_*-\widetilde{\psi}^\dagger_d$ on $\partial S^d\setminus G^d(1)$ and vanishes elsewhere on $\partial \F^d$.
    In this far-field regime, the fast varying inner cutoff $\chi_1$ in \cref{eq_def_psi_tilde_d} vanishes and $\widetilde{\psi}_\text{gap}$ and its derivatives are uniformly bounded.
    This eliminates all singular $d$-dependence, yielding $\|\psi_\text{BC}^\text{far}\|_{\dot{H}^2(\F^d)} = \bigo(1)$.
    Summing these components yields a valid lifting $\psi_\text{BC}$.

    Consequently, the functional $\phi \mapsto \int_{\F^d} \Delta \psi_\text{BC} \Delta \phi$ is uniformly continuous on $H^2_0(\F^d)$, which establishes 
    \begin{equation}\label{eq_bound_psierr}
        \|\psi^{\text{err}}\|_{\dot{H}^2(\F^d)} = \bigo(1).
    \end{equation}
    Together with \cref{eq_diff_gag}, this yields $\|\opsi_d-\widetilde{\psi}_d\|_{\dot{H}^2(\F^d)} \leq \|\psi^{\text{err}}\|_{\dot{H}^2(\F^d)} + \|\widetilde{\psi}^\dagger_d-\widetilde{\psi}_d\|_{\dot{H}^2(\F^d)}= \bigo(1)$.
\end{proof}
\begin{proof}[Proof of \cref{prop_stokes_aux}]
    We first detail the definition and energy of the ansatz $\widetilde{\psi}_d$ from \cref{thm_stokes}.
    By definition \cref{eq_def_psi_tilde_d}, the approximation $\widetilde{\psi}_d$ coincides with the gap profile $\widetilde{\psi}_{\text{gap}}$ on the transition region $A^\text{tr}=A^d(dr_0,1)$, yielding \cref{eq_1_th_stokes}--\cref{eq_3_th_stokes}.
    Consequently, \cref{lem_estim_tildepsir_bilapl} establishes \cref{eq_4_th_stokes}.
    
    Since $A^\text{tr}\subset \F^d$, the local error bound \cref{eq_1_th_stokes_aux} follows from \cref{lem_estim_error_approx_bilapl}.
    Split the outer domain $\Omega \setminus A^\text{tr}$ into $S^d$ and $\F^d \setminus A^\text{tr}$ to obtain
    \begin{equation*}
        \|\opsi_d\|_{\dot{H}^2(\Omega\setminus A^\text{tr})} \leq \|\opsi_d\|_{\dot{H}^2(S^d)} + \|\widetilde{\psi}_d\|_{\dot{H}^2(\F^d\setminus A^\text{tr})} +\|\opsi_d - \widetilde{\psi}_d\|_{\dot{H}^2(\F^d\setminus A^\text{tr})}.
    \end{equation*}
    Extending $\opsi_d$ by the rigid body velocity $\psi_*$, the first term is bounded independently of $d$.
    The second and third terms are controlled by \cref{lem_estim_ext_psid} and \cref{lem_estim_error_approx_bilapl}, respectively.
    This establishes \cref{eq_2_th_stokes_aux}.

    Let $\widetilde{\psi}^\dagger_d$, the error $\psi^\text{err} = \opsi_d - \widetilde{\psi}^\dagger_d$, and the lifting $\psi_\text{BC}$ be given as in the preceding lemma.
    We expand the squared norm
    \begin{equation*}
        \|\opsi_d\|^2_{\dot{H}^2(\F^d)}  = \|\widetilde{\psi}^\dagger_d\|^2_{\dot{H}^2(\F^d)} + 2 \int_{\F^d}\nabla^2\widetilde{\psi}^\dagger_d:\nabla^2\psi^\text{err} + \|\psi^\text{err}\|^2_{\dot{H}^2(\F^d)}.
    \end{equation*}
    The first term behaves as $\|\widetilde{\psi}_d\|^2_{\dot{H}^2(A^\text{tr})} + \bigo(1)$ via \cref{eq_diff_gag} and \cref{lem_estim_ext_psid}.
    The third is controlled by \cref{eq_bound_psierr}.
    For the cross term, decomposing the error into $\psi^\text{err} = \psi^\text{err}_{0} + \psi_\text{BC}$ yields
    \begin{equation*}
        \int_{\F^d}\nabla^2\widetilde{\psi}^\dagger_d:\nabla^2\psi^\text{err} = \int_{\F^d}\nabla^2\opsi_d :\nabla^2\psi^\text{err}_0 + \int_{\F^d}\nabla^2\opsi_d :\nabla^2\psi_\text{BC}  - \|\psi^\text{err}\|^2_{\dot{H}^2(\F^d)}.
    \end{equation*}
    The first integral vanishes by the weak formulation of the auxiliary Stokes problem \cref{eq_stream_stokes_aux}, while the last term is uniformly bounded by \cref{eq_bound_psierr}.
    The second integral involves the lifting $\psi_\text{BC}$, which is supported in $\F^d \setminus A^d(2dr_0, \frac{1}{2})$ and uniformly bounded in $\dot{H}^2$ by \cref{eq:bound_lifting_psi}.
    Observe that
    \begin{equation*}
        \|\opsi_d\|_{\dot{H}^2(\F^d\setminus A^d(2dr_0,\frac{1}{2}))}^2 \leq \|\opsi_d\|^2_{\dot{H}^2(\F^d\setminus A^\text{tr})} + \|\widetilde{\psi}_d\|^2_{\dot{H}^2(A^\text{tr}\setminus A^d(2dr_0,\frac{1}{2}))} + \|\opsi_d - \widetilde{\psi}_d\|^2_{\dot{H}^2(A^\text{tr}\setminus A^d(2dr_0,\frac{1}{2}))}.
    \end{equation*}
    The terms on the right-hand side are bounded by \cref{eq_2_th_stokes_aux}, explicit computation (analogous to \cref{lem_estim_tildepsir_bilapl} on $A_\text{rem}$), and \cref{eq_1_th_stokes_aux}, respectively.
    Consequently, the cross-term is uniformly bounded, establishing \cref{eq_3_th_stokes_aux}.

    It remains to verify the uniform $\dot{H}^1$-bound.
    We decompose the norm as follows
    \begin{equation*}
        \|\opsi_d\|_{\dot{H}^1(\Omega)} \leq\|\psi_*\|_{\dot{H}^1(S^d)} + \|\widetilde{\psi}_d\|_{\dot{H}^1(\F^d)} + \|\opsi_d - \widetilde{\psi}_d\|_{\dot{H}^1(\F^d)}.
    \end{equation*}
    The first two terms are controlled uniformly via the explicit expression for $\psi_*$ and the bounds \cref{eq_bound_profil_gap_psi}--\cref{eq_bound_profil_gap_psiconst}.
    In particular, for $\nabla\widetilde{\psi}_d^c$, the support is included in $A^d(\lambda dr_0,2)$, making it uniformly bounded.
    For the third term, $\opsi_d - \widetilde{\psi}_d$ belongs to $H^2(\F^d)$ and, according to \cref{eq_stream_stokes_aux} and \cref{eq_syst_local_psid}, has a vanishing normal trace on the exterior boundary $\partial\Omega$.
    Poincaré's inequality yields
    \begin{equation*}
        \|\opsi_d - \widetilde{\psi}_d\|_{\dot{H^1}(\F^d)} \leq C(\Omega) \|\opsi_d - \widetilde{\psi}_d\|_{\dot{H}^2(\F^d)} = \bigo(1).
    \end{equation*}
    This establishes the global estimates \cref{eq_4_th_stokes_aux} and concludes the proof.
\end{proof}
\subsection{Return to the original Stokes problem}\label{sec_return_init_stokes_pb}
\subsubsection{Asymptotic behavior of $C_*$}\label{sec_return_init_stokes_pb_cst}
The stream function problem \cref{eq_stream_stokes} requires $\psi_d$ to satisfy the compatibility condition $\int_{\partial S^d} \partial_n \Delta \psi_d = 0$.
Substituting the decomposition \cref{eq_decomp_Cstar} into this relation yields
\begin{equation}\label{eq_expre_Cstar}
    C_* = - d \frac{\int_{\partial S^d} \partial_n \Delta \opsi_d}{\int_{\partial S^d} \partial_n \Delta \opsi_d^c}.
\end{equation}
Before establishing \cref{lem_constant_stokes}, we first prove the following technical estimate.
\begin{lemma}\label{lem_main_order_stokes_entry_aux}
    Let $\opsi_1$ and $\opsi_2$ be solutions to the auxiliary Stokes problem \cref{eq_stream_stokes_aux} for arbitrary boundary data $\psi_*$ given in \cref{eq_elem_veloc}, and let $\widetilde{\psi}_1$ and $\widetilde{\psi}_2$ denote their respective approximations from \cref{prop_stokes_aux}.
    Then,
    \begin{equation*}
        \int_{\Omega} \nabla^2 \opsi_1 : \nabla^2 \opsi_2 = \int_{A^\text{tr}} \nabla^2 \widetilde{\psi}_1 : \nabla^2 \widetilde{\psi}_2 + \bigo(1).
    \end{equation*} 
\end{lemma}
\begin{proof}
    For $i=1,2$, by construction $\opsi_i = \psi_*$ on $S^d$, allowing us to restrict our analysis to the fluid domain $\F^d$, as the contribution from $S^d$ is uniformly bounded.
    Recall that the approximation profiles $\widetilde{\psi}_i$ are defined on the entire domain $\F^d$ by \cref{eq_def_psi_tilde_d}.
    Introducing the error $\psi^{\mathrm{err}}_i := \opsi_i - \widetilde{\psi}_i$, we decompose the integral as
    \begin{equation*}
        \int_{\F^d} \nabla^2 \opsi_1 : \nabla^2 \opsi_2 = \int_{\F^d} \nabla^2 \widetilde{\psi}_1 : \nabla^2 \widetilde{\psi}_2 +\int_{\F^d} \nabla^2 \psi^\text{err}_1 : \nabla^2 \opsi_2 +\int_{\F^d} \nabla^2 \widetilde{\psi}_1 : \nabla^2 \psi^\text{err}_2.
    \end{equation*}
    By \cref{lem_estim_ext_psid}, the first term is uniformly bounded in $\Omega\setminus A^\text{tr}$.
    It remains to show that the remaining terms are bounded independently of $d$.

    Following the approach of \cref{lem_estim_error_approx_bilapl}, we decompose the error as $\psi^\text{err}_i = \psi^\text{err}_{i,0} + \psi^\text{BC}_i$, where $\psi^\text{BC}_i$ is the boundary lifting constructed such that $\psi^\text{err}_{i,0} \in \dot{H}^2_0(\mathcal{F}^d)$.
    To estimate the second term, $\int_{\F^d} \nabla^2 \psi^\text{err}_1 : \nabla^2 \opsi_2$, we note that since $\psi^\text{err}_{1,0} \in \dot{H}^2_0(\mathcal{F}^d)$, its $\dot{H}^2$-inner product with $\opsi_2$ vanishes by the weak formulation of the Stokes problem for $\opsi_2$.
    Thus, there remains only the contribution of the lifting $\psi^\text{BC}_1$, which is supported in $\mathcal{F}^d \setminus A^d(2dr_0, \frac{1}{2})$ and uniformly bounded by \cref{eq:bound_lifting_psi}.
    It remains to show that the right-hand side of
    \begin{equation*}
        \|\opsi_2\|_{\dot{H}^2(\F^d\setminus A^d(2dr_0,\frac{1}{2}))}^2 \leq \|\opsi_2\|^2_{\dot{H}^2(\F^d\setminus A^\text{tr})} + \|\widetilde{\psi}_2\|^2_{\dot{H}^2(A^\text{tr}\setminus A^d(2dr_0,\frac{1}{2}))} + \|\opsi_2 - \widetilde{\psi}_2\|^2_{\dot{H}^2(A^\text{tr}\setminus A^d(2dr_0,\frac{1}{2}))}
    \end{equation*}
    is uniformly bounded.
    These terms are controlled by \cref{eq_2_th_stokes_aux}, an explicit computation (as in \cref{lem_estim_tildepsir_bilapl} for $A_\text{rem}$), and \cref{eq_1_th_stokes_aux}, respectively.

    An analogous argument applies to the final term.
    The contribution involving the lifting $\psi^\text{BC}_2$ is bounded identically using \cref{lem_estim_ext_psid}.
    For the component involving $\psi^\text{err}_{2,0}$, the continuity relation \cref{eq_cont_psi_d_H20} and the error estimates from \cref{lem_estim_error_approx_bilapl} yield
    \begin{equation*}
        \left|\int_{\F^d} \nabla^2 \widetilde{\psi}_1 : \nabla^2 \psi^\text{err}_{2,0}\right| \leq C\|\nabla^2\psi^\text{err}_{2,0}\|_{L^2(\F^d)} = \bigo(1).
    \end{equation*}
\end{proof}
\begin{proof}[Proof of \cref{lem_constant_stokes}]
    By definition, $\opsi_d^c = d$ on $\partial S^d$ and $\opsi_d^c = 0$ on $\partial \Omega$.
    Since $\Delta^2 \opsi_d = 0$ in $\F^d$ and $\partial_n \opsi_d^c = 0$ on $\partial \F^d$, applying Green's identity yields
    \begin{equation*}
        d\int_{\partial S^d} \partial_n \Delta \opsi_d = \int_{\partial \F^d} \opsi_d^c \partial_n \Delta \opsi_d = \int_{\F^d} \nabla\opsi_d^c \cdot \nabla \Delta \opsi_d= - \int_{\F^d} \Delta\opsi_d^c \Delta \opsi_d.
    \end{equation*}
    Consequently, relation \cref{eq_expre_Cstar} can be rewritten as
    \begin{equation}\label{eq_exp_cstar}
        C_* = - d \frac{\int_{\F^d} \Delta \opsi_d^c \Delta \opsi_d}{\int_{\F^d} |\Delta \opsi_d^c|^2}.
    \end{equation}
    By \cref{eq_stream_stokes_aux}, the extension of $\opsi_d$ to $S^d$ via $\psi_*$ belongs to $H^2_0(\Omega)$. 
    Since $\Delta\psi_*^c=0$ in $S^d$, the integral over the fluid domain $\F^d$ extends to the entire domain $\Omega$.
    Furthermore, since the extended functions $\opsi_d$ and $\opsi_d^c$ both reside in $H^2_0(\Omega)$, their traces and normal derivatives vanish on the global boundary $\partial \Omega$.
    This permits a double integration by parts, which, combined with \cref{lem_main_order_stokes_entry_aux}, yields
    \begin{equation*}
        \int_{\F^d} \Delta \opsi_d^c \Delta \opsi_d = \int_{\Omega} \Delta \opsi_d^c \Delta \opsi_d = \int_{\Omega} \nabla^2 \opsi_d^c :  \nabla^2 \opsi_d = \int_{A^\text{tr}} \nabla^2 \widetilde{\psi}_d^c : \nabla^2 \widetilde{\psi}_d + \bigo(1).
    \end{equation*}
    We may explicitly estimate this principal term using the techniques from the proof of \cref{lem_estim_tildepsir_bilapl}.
    First decompose $A^\text{tr}$ into analogous right and left components, $A^\text{tr}_R$ and $A^\text{tr}_L$. 
    From definitions \cref{eq_def_psi_tilde_d} and \cref{eq_def_psi_gap}, we have
    \begin{equation*}
        \int_{A^\text{tr}_R} \nabla^2 \widetilde{\psi}_d^c : \nabla^2 \widetilde{\psi}_d = d^{-2}\int_{A^\text{tr}_R-x_R} \big(\nabla^2 \psi_R^c : \nabla^2 \psi_R\big)(d^{-1}x),
    \end{equation*}
    where $\psi_R$ and $\psi_R^c$ are defined in \cref{lem_expr_psiL_psiR}.
    To estimate this integral, suppose for instance that $\opsi_d=\opsi_d^\perp$, which implies $\psi_R(r,\theta) = \psi_R^\perp(r,\theta) = r f_R^\perp(\theta)$.
    We partition the shifted region $A^\text{tr}_R - x_R$ into a principal domain $A_\text{main} := \{(r,\theta) : r\in(dr_1,1),~\theta\in(0,\theta_R)\}$, for a sufficiently large radius $r_1>0$ depending only on $\theta_L$ and $\theta_R$, and a remainder $A_\text{rem}$.
    Computing the integral over $A_\text{main}$ yields
    \begin{equation*}
        d^{-2}\int_{A_\text{main}} \big(\nabla^2 \psi_R^c : \nabla^2 \psi_R\big)(d^{-1}x)=d^{-2}\int_{\theta=0}^{\theta_R}\int_{r=dr_1}^{1} \frac{d^3}{r^3}{f_R^c} ''({f_R^\perp} + {f_R^\perp} '')rdrd\theta = \bigo(1).
    \end{equation*}
    As $A_\text{rem}$ is confined to a left component whose radii scale with $d$ and a macroscopic right component with fixed radii, its energy contribution is of order $\bigo(1)$.
    Analogous bounds hold for the left component $A^\text{tr}_L$ and for the other profiles $\opsi_d^\|$, $\opsi_d^\circlearrowleft$, and $\opsi_d^c$.
    Combining these bounds, we deduce that
    \begin{equation}\label{eq_estim_terme_croise}
        \int_{\F^d} \Delta \opsi_d^c \Delta \opsi_d = \int_{\Omega} \nabla^2 \opsi_d^c : \nabla^2 \opsi_d = \bigo(1).
    \end{equation}
    In particular, the corresponding computation for the constant case $\psi_*^c$ shows that the integral over $A_\text{main}$ is strictly of order one, and therefore $\|\opsi_d^c\|_{\dot{H}^2(\F^d)}^2=\int_{\F^d} |\Delta \opsi_d^c|^2$ is as well.
    We deduce from \cref{eq_exp_cstar} that $C_* = \bigo(d)$, completing the proof.
\end{proof}
\subsubsection{Error estimates for $\psi_d$}
Combining \cref{lem_constant_stokes} with the bound $\|\opsi_d^c\|_{\dot{H}^2(\F^d)}=\bigo(1)$ yields the key relation
\begin{equation}\label{eq_lien_aux_main}
    \|\psi_d - \opsi_d\|_{\dot{H}^2(\Omega)} = C_* d^{-1} \|\opsi_d^c\|_{\dot{H}^2(\Omega)} = \bigo(1).
\end{equation}
This uniform bound ensures that the auxiliary Stokes solution $\opsi_d$ well approximates the true solution $\psi_d$.
To conclude the proof of \cref{thm_stokes}, it merely remains to couple this result with \cref{prop_approx_stokes_tilde}, which established $\widetilde{\psi}_d$ as a valid approximation for $\opsi_d$.
\begin{proof}[Proof of \cref{thm_stokes}]
    Relations \cref{eq_1_th_stokes}--\cref{eq_4_th_stokes} have already been established in the proof of \cref{prop_stokes_aux}.
    
    To establish \cref{eq_5_th_stokes}, we use the triangle inequality to bound
    \begin{equation*}
        \|\psi_d - \widetilde{\psi}_d \|^2_{\dot{H}^2(A^\text{tr})}\leq 2 \Big( \|\psi_d - \opsi_d \|^2_{\dot{H}^2(A^\text{tr})} + \|\opsi_d - \widetilde{\psi}_d \|^2_{\dot{H}^2(A^\text{tr})}\Big).
    \end{equation*}
    The first term is bounded by \cref{eq_lien_aux_main} and the second by \cref{eq_1_th_stokes_aux}.
    The estimate \cref{eq_6_th_stokes} follows analogously using \cref{eq_2_th_stokes_aux}.

    To prove \cref{eq_7_th_stokes}, we expand 
    \begin{equation*}
        \|\psi_d\|^2_{\dot{H}^2(\F^d)}  = \|\opsi_d\|^2_{\dot{H}^2(\F^d)} + 2 \int_{\F^d}\nabla^2\opsi_d:\nabla^2(\psi_d - \opsi_d) + \|\psi_d - \opsi_d\|^2_{\dot{H}^2(\F^d)}.
    \end{equation*}
    The first term is controlled via \cref{eq_3_th_stokes_aux}, while the third is uniformly bounded by \cref{eq_lien_aux_main}.
    Next, we rewrite the cross-term as
    \begin{equation*}
        \int_{\F^d}\nabla^2\opsi_d:\nabla^2(\psi_d - \opsi_d) = C_*d^{-1}\int_{\F^d}\nabla^2\opsi_d:\nabla^2\opsi_d^c,
    \end{equation*}
    which is uniformly bounded by \cref{lem_constant_stokes} and \cref{eq_estim_terme_croise}, thereby establishing \cref{eq_7_th_stokes}.

    Finally, to prove \cref{eq_8_th_stokes}, we bound
    \begin{equation*}
        \|\psi_d \|^2_{\dot{H}^1(\Omega)} \leq 2\big( \|\psi_d - \opsi_d \|^2_{\dot{H}^1(\Omega)} + \|\opsi_d\|^2_{\dot{H}^1(\Omega)}\big).
    \end{equation*}
    Since $\psi_d - \opsi_d$ has zero normal trace on the exterior boundary $\partial\Omega$, the first term is bounded using Poincaré's inequality and \cref{eq_lien_aux_main}.
    The second term is bounded by \cref{eq_4_th_stokes_aux}, completing the proof.
\end{proof}

\section{The Stokes resistance matrix}\label{sec_Stokes_matrix}
Building upon \cref{thm_stokes}, we analyze the asymptotic behavior of the Stokes resistance matrix as $d \to 0$.
In particular, we establish \cref{thm_stokes_matrices}, which provides a first-order expansion of this operator.
By mapping the rigid kinematics of a particle to the fluid hydrodynamic response, the resistance matrix encodes the physical constraints imposed by the no-slip boundary condition.
Characterizing this operator is essential for deriving effective models, since macroscopic dissipation emerges directly from averaging these local resistance properties (see \cite{allaire1991a,hillairet_effect_2019, desvillettes2008,gerard-varet2022} for further details).

Let $V \in \R^2$ and $\omega \in \R$ denote the translational and rotational velocities of the solid $S^d$ with respect to the origin.
We denote by $(u_d(V,\omega), p_d(V,\omega))$ the unique solution to the Stokes problem \cref{eq_stokes} subject to the rigid boundary condition $u_* =V + \omega x^\perp$ on $\partial S^d$.
Centering rotations at the origin ensures consistency with the elementary velocities defined in \cref{eq_cano_motions}.
An arbitrary center of rotation can be recovered via \cref{eq_rot_orgin}.

Let $\Sigma(V,\omega) = 2D(u_d(V,\omega)) - p_d(V,\omega)\mathbb{I}_2$ be the associated stress tensor.
The hydrodynamic force $F_d(V,\omega) \in \R^2$ and torque $T_d(V,\omega) \in \R$ exerted by the fluid on the inclusion $S^d$ are given by
\begin{equation*}
    F_d(V,\omega) = \int_{\partial S^d} \Sigma(V,\omega)n d\sigma, \qquad T_d(V,\omega) = \int_{\partial S^d} x^\perp \cdot \Sigma(V,\omega)n d\sigma,
\end{equation*}
where $n$ denotes the unit inward normal to $\partial S^d$.
By the linearity of the Stokes system, we define the Stokes resistance matrix $\mathbb{A}_d \in \mathcal{M}_3(\R)$ such that
\begin{equation}\label{def_matrice_stokes}
    \begin{pmatrix}
        F_d(V,\omega)\\
        T_d(V,\omega)
    \end{pmatrix} = \mathbb{A}_d\begin{pmatrix}
        V\\
        \omega
    \end{pmatrix}.
\end{equation}
To study the entries of $\mathbb{A}_d$, we introduce the basis of elementary rigid motions $(u^*_i)_{1\leq i\leq 3}$ corresponding to the one defined in \cref{eq_cano_motions}.
Letting $(u_i, p_i)$ and $\Sigma_i$ be the associated quantities, the entries of the resistance matrix take the form
\begin{equation*}
    \mathbb{A}_d = \left(\int_{\partial S^d} \Sigma_i n \cdot u^*_j d\sigma\right)_{1\leq i,j\leq 3}.
\end{equation*}
For $i \in \{1,2\}$, the coefficient $(\mathbb{A}_d)_{i,j}$ represents the force exerted on $S^d$ in the direction $e_i$ induced by the rigid motion $u^*_j$, while $(\mathbb{A}_d)_{3,j}$ corresponds to the resulting torque.
Integrating by parts, we can express the resistance matrix as
\begin{equation*}
    \mathbb{A}_d = \left(\int_{\Omega} \nabla u_i : \nabla u_j \right)_{1\leq i,j\leq 3}.
\end{equation*}
This representation allows for a direct application of \cref{thm_stokes}.
For the diagonal coefficients, \cref{eq_7_th_stokes} immediately yields
\begin{equation}\label{eq_rel_diag_coeff}
    (\mathbb{A}_d)_{i,i} = \|u_i\|_{\dot{H}^1(\Omega)}^2 =  \|\widetilde{u}_i\|_{\dot{H}^1(A^\text{tr})}^2 + \bigo(1),
\end{equation}
where the leading-order term on the right-hand side is explicitly known.

Although applying \cref{thm_stokes} directly to the off-diagonal terms yields an $\bigo(\sqrt{|\log(d)|})$ error bound, \cref{lem_main_order_stokes_entry} establishes a sharper $\bigo(1)$ estimate.
\begin{lemma}\label{lem_main_order_stokes_entry}
    For any $1 \leq i, j \leq 3$, as $d\to 0$, we have
    \begin{equation*}
        \int_{\Omega} \nabla u_i : \nabla u_j = \int_{A^\text{tr}} \nabla \widetilde{u}_i : \nabla \widetilde{u}_j + \bigo(1).
    \end{equation*}
\end{lemma}
\begin{proof}
    Let $\psi_i$ and $\psi_j$ denote the stream functions associated with $u_i$ and $u_j$, respectively.
    Using analogous notation for the auxiliary quantities, we decompose the integral as
    \begin{equation*}
        \int_{\Omega} \nabla^2 \psi_i : \nabla^2 \psi_j = \int_{\Omega} \nabla^2 \opsi_i : \nabla^2 \opsi_j + \int_{\F^d}\nabla^2\opsi_i:\nabla^2(\psi_j - \opsi_j) + \int_{\F^d}\nabla^2(\psi_i- \opsi_i):\nabla^2\psi_j.
    \end{equation*}
    By \cref{lem_main_order_stokes_entry_aux}, the first term on the right-hand side yields the desired principal part.
    It therefore suffices to show that the remaining two cross-terms are uniformly bounded.
    For the second term, substituting $\psi_j = \opsi_j + C_*d^{-1}\opsi_d^c$ yields
    \begin{equation*}
        \int_{\F^d}\nabla^2\opsi_i:\nabla^2(\psi_j - \opsi_j) = C_*d^{-1}\int_{\F^d}\nabla^2\opsi_i:\nabla^2\opsi_d^c,
    \end{equation*}
    which is uniformly bounded by \cref{lem_constant_stokes} and \cref{eq_estim_terme_croise}.
    Similarly, expanding the final term gives
    \begin{equation*}
        \int_{\F^d}\nabla^2(\psi_i- \opsi_i):\nabla^2\psi_j = \int_{\F^d}\nabla^2(\psi_i- \opsi_i):\nabla^2\opsi_j + C_*d^{-1} \int_{\F^d}\nabla^2(\psi_i- \opsi_i):\nabla^2\opsi_d^c.
    \end{equation*}
    Both components are bounded independently of $d$ by the same argument.
\end{proof}
\begin{theorem}[First-order expansion of the Stokes resistance matrix]\label{thm_stokes_matrices}
    As $d \to 0$, the entries of the Stokes resistance matrix $\mathbb{A}_d$ defined in \cref{def_matrice_stokes} satisfy
    \begin{subequations}
    \begin{align}
        (\mathbb{A}_d)_{1,1} & = 2\left[\frac{\theta_L - \sin(\theta_L)\cos(\theta_L)}{\theta_L^2 - \sin^2(\theta_L)} + \frac{\theta_R - \sin(\theta_R)\cos(\theta_R)}{\theta_R^2 - \sin^2(\theta_R)}\right]|\log(d)| +\bigo(1),\label{entry_stokes_mat_11}\\[0.5em]
        (\mathbb{A}_d)_{2,2} & = 2\left[\frac{\theta_L + \sin(\theta_L)\cos(\theta_L)}{\theta_L^2 - \sin^2(\theta_L)} + \frac{\theta_R + \sin(\theta_R)\cos(\theta_R)}{\theta_R^2 - \sin^2(\theta_R)}\right]|\log(d)| + \bigo(1),\label{entry_stokes_mat_22}\\[0.5em]
        (\mathbb{A}_d)_{1,2} & = (\mathbb{A}_d)_{2,1} = 2 \left[\frac{\sin^2(\theta_L)}{\theta_L^2 - \sin^2(\theta_L)} - \frac{\sin^2(\theta_R)}{\theta_R^2 - \sin^2(\theta_R)}\right]|\log(d)| + \bigo(1),\label{entry_stokes_mat_12}\\[0.5em]
        (\mathbb{A}_d)_{i,3} & = (\mathbb{A}_d)_{3,i} = \bigo(1), \qquad i=1,2,3.\label{entry_stokes_mat_i3}
    \end{align}
    \end{subequations}
\end{theorem}
\begin{proof}
    The expressions for the diagonal terms follow directly from \cref{eq_rel_diag_coeff} and \cref{eq_4_th_stokes}.
    For the off-diagonal terms, by \cref{lem_main_order_stokes_entry}, it suffices to compute, for $i<j$
    \begin{equation*}
        \int_{A^\text{tr}} \nabla \widetilde{u}_i : \nabla \widetilde{u}_j = \int_{A^\text{tr}} \nabla^2 \widetilde{\psi}_i : \nabla^2 \widetilde{\psi}_j.
    \end{equation*}
    The remaining cases $i>j$ follow by symmetry.
    We first evaluate these integrals on the right side $A^\text{tr}_R$.
    Following the decomposition strategy of \cref{lem_estim_tildepsir_bilapl}, we partition the shifted domain as $A_R^\text{tr}-x_R = A_\text{main} \cup dA_\text{rem}$, where $A_\text{main} = \{(r,\theta) \mid r \in (dr_1, 1), \theta \in (0,\theta_R)\}$ is the principal angular sector, and $r_1>0$ is independent of $d$.
    Since $A_\text{rem}$ is confined to a fixed-radius outer region and a bounded $d$-scaled inner region, its contribution is of order $\bigo{O}(1)$ for any pair $(i,j)$.
    Consequently, we obtain
    \begin{equation*}
        \int_{A^\text{tr}_R} \nabla \widetilde{u}_i : \nabla \widetilde{u}_j = \int_{A_\text{main}} \big[\nabla^2 \widetilde{\psi}_i : \nabla^2 \widetilde{\psi}_j\big](x+x_R) \,dx+ \bigo(1).
    \end{equation*}
    For $i=1$ and $j=2$, \cref{eq_1_th_stokes} and \cref{eq_2_th_stokes} yield
    \begin{equation*}
        \int_{A^\text{tr}_R} \nabla^2 \widetilde{\psi}_d^\| : \nabla^2 \widetilde{\psi}_d^\perp = \int^{\theta_R}_{0}\int_{dr_1}^{1}\frac{\big(f_R^\|(\theta) + {f_R^\|}''(\theta)\big)\big(f_R^\perp(\theta) + {f_R^\perp}''(\theta)\big)}{r^2}rdrd\theta + \bigo(1).
    \end{equation*}
    Substituting the expressions for the angular functions from \cref{lem_expr_psiL_psiR} gives
    \begin{equation*}
        \int_{A^\text{tr}_R} \nabla^2 \widetilde{\psi}_d^\| : \nabla^2 \widetilde{\psi}_d^\perp = -2D_R^\perp |\log(d)| + \bigo(1).
    \end{equation*}
    A fully analogous argument applies to the left domain $A^\text{tr}_L$.
    Integrating over the corresponding angular sector $(\pi-\theta_L,\pi)$ yields
    \begin{equation*}
        \int_{A^\text{tr}_L} \nabla^2 \widetilde{\psi}_d^\| : \nabla^2 \widetilde{\psi}_d^\perp = 2D_L^\perp |\log(d)| + \bigo(1).
    \end{equation*}
    Summing these contributions leads to \cref{entry_stokes_mat_12}.
    Finally, for $i=1,2$ and $j=3$, \cref{eq_1_th_stokes}--\cref{eq_3_th_stokes} give
    \begin{equation*}
        \int_{A^\text{tr}_R} \nabla^2 \widetilde{\psi}_i : \nabla^2 \widetilde{\psi}_d^\circlearrowleft = \int^{\theta_R}_{0}\int_{dr_1}^{1}\frac{\big(f_R^i(\theta) + {f_R^i}''(\theta)\big)\big(f_R^\circlearrowleft(\theta) + \frac{1}{2}{f_R^\circlearrowleft}''(\theta)\big)}{r}rdrd\theta + \bigo(1)=\bigo(1).
    \end{equation*}
    The same result holds on $A^\text{tr}_L$.
    This establishes \cref{entry_stokes_mat_i3} for $i\in\{1,2\}$.
\end{proof}

\section{The pressure field}\label{sec_pressure}
With the asymptotic velocity profile $u_d$ established in \cref{thm_stokes}, we now derive a leading-order approximation for the associated pressure $p_d$
The natural functional setting for the pressure is the quotient space $L^2_0(\Omega) = L^2(\Omega)/\R$, whose elements are square-integrable functions defined up to an additive constant.
It is equipped with the norm
\begin{equation*}
    \|f\|_{L^2_0(\Omega)} = \inf_{c\in\R} \|f + c\|_{L^2(\Omega)} \equiv \|f - \langle f \rangle_\Omega \|_{L^2(\Omega)},
\end{equation*}
where $\langle f \rangle_\Omega = \frac{1}{|\Omega|}\int_\Omega f$ denotes the mean value of $f$ over $\Omega$.

Because the local velocity field $\widetilde{u}_\text{gap}=\nabla^\perp\widetilde{\psi}_\text{gap}$ from \cref{eq_def_psi_gap} captures the leading-order behavior of $u_d$, its associated pressure is a natural candidate to approximate $p_d$.
\begin{proposition}[Local pressure profile]\label{prop_pressure}
    For the elementary boundary data $u_*^\perp$, $u_*^\|$, and $u_*^\circlearrowleft$, the corresponding local pressure profiles are given by
    \begin{subequations}
    \begin{align}
        \widetilde{p}_\text{gap}^{\,\perp}(x) &=\begin{cases}
                \frac{1}{|x-x_L|}\big(f_L^\perp{}'(\arg(x-x_L)) + f_L^\perp{}'''(\arg(x-x_L))\big) \quad& \text{in } A^d_L(\lambda d r_0,2),\\[0.5em]
                \frac{1}{|x-x_R|}\big(f_R^\perp{}'(\arg(x-x_R)) + f_R^\perp{}'''(\arg(x-x_R))\big) \quad& \text{in } A^d_R(\lambda d r_0,2),
            \end{cases} \label{eq_1_th_pression} \\[0.5em]
        \widetilde{p}_\text{gap}^{\,\|}(x) &=\begin{cases}
                \frac{1}{|x-x_L|}\big(f_L^\|{}'(\arg(x-x_L)) + f_L^\|{}'''(\arg(x-x_L))\big)\quad& \text{in } A^d_L(\lambda d r_0,2), \\[0.5em]
                \frac{1}{|x-x_R|}\big(f_R^\|{}'(\arg(x-x_R)) + f_R^\|{}'''(\arg(x-x_R))\big) \quad& \text{in } A^d_R(\lambda d r_0,2),
            \end{cases}\label{eq_2_th_pression}\\[0.5em]
        \widetilde{p}_\text{gap}^{\,\circlearrowleft}(x) &=\begin{cases}
                -\frac{1}{2}\ln(|x-x_L|)\big(4f_L^\circlearrowleft{}'(\arg(x-x_L))+f_L^\circlearrowleft{}'''(\arg(x-x_L))\big),\quad &\text{in } A^d_L(\lambda d r_0,2),\\[0.5em]
                -\frac{1}{2}\ln(|x-x_R|)\big(4f_R^\circlearrowleft{}'(\arg(x-x_R))+f_R^\circlearrowleft{}'''(\arg(x-x_R))\big),\quad &\text{in } A^d_R(\lambda d r_0,2),
            \end{cases}\label{eq_3_th_pression}
    \end{align}
    where the angular functions $f_L$ and $f_R$ are specified in \cref{lem_expr_psiL_psiR}.
    These profiles satisfy
    \begin{equation*}
        -\Delta \widetilde{u}_{\text{gap}} + \nabla \widetilde{p}_{\text{gap}} =0 \quad \text{in} \quad A^d(\lambda d r_0,2).
    \end{equation*}
    \end{subequations}
\end{proposition}
These local profiles characterize the global asymptotic behavior of the pressure.
\begin{theorem}[Asymptotic behavior of the pressure]\label{thm_pression}
    There exists $\widetilde{p}_d \in C^\infty(\F^d)$ such that $\widetilde{p}_d = \widetilde{p}_\text{gap}$ on $A^\text{tr}$ and, as $d \to 0$,
    \begin{subequations}
    \begin{equation}\label{eq_4_th_pression}
        \begin{cases}
        \|\widetilde{p}_d^{\,\perp}\|_{L^2(A^\text{tr})}^2 & = \big[4(B_R^\perp - B_L^\perp) + 2(D_R^\|-D_L^\|) + 8 (B_R^\perp D_R^\perp - B_L^\perp D_L^\perp)\big]|\log(d)| + \bigo(1),\\[0.5em] 
        \|\widetilde{p}_d^{\,\|}\|_{L^2(A^\text{tr})}^2 & = \big[4(D_L^\|-D_R^\|) + 2(B_L^\perp-B_R^\perp) + 8 (B_R^\| D_R^\|-B_L^\| D_L^\|)\big]|\log(d)| +\bigo(1),\\[0.5em] 
        \|\widetilde{p}_d^{\,\circlearrowleft}\|_{L^2(A^\text{tr})}^2 & = \bigo(1),\\[0.5em] 
        \|\widetilde{p}_d\|_{L^2(\F^d\setminus A^\text{tr})}^2 &= \bigo(1),
        \end{cases}
    \end{equation}
    where the constants $B$ and $D$ are specified in \cref{lem_expr_psiL_psiR}.
    Moreover, for any elementary boundary condition, this global profile satisfies
    \begin{equation}\label{eq_5_th_pression}
        \|p_d - \widetilde{p}_d\|^2_{L^2_0(\F^d)} \leq C(\theta_L,\theta_R).
    \end{equation}
    \end{subequations}
\end{theorem}
\begin{proof}[Proof of \cref{prop_pressure}]
    By \cref{eq_syst_local_psigap}, the stream function $\widetilde{\psi}_\text{gap}$ is biharmonic in $A^d(\lambda dr_0,2)$.
    By the standard velocity-pressure reconstruction, De Rham's theorem guarantees that $\Delta \widetilde{u}_{\text{gap}}$ is a pure gradient in $A^d(\lambda dr_0,2)$, ensuring the existence of an associated local pressure field $\widetilde{p}_{\text{gap}}$ in this region.
    
    To compute $\widetilde{p}_{\text{gap}}$ explicitly, we introduce shifted polar coordinates $(r_s,\theta_s)$ centered at $x_R$, defined by $r_s= |x-x_R|$ and $\theta_s = \arg(x-x_R)$.
    In the right subregion $A^d_R(\lambda dr_0,2)$, the Stokes equation reads
    \begin{equation*}
        \Delta \widetilde{u}_{\text{gap}}(x) =\nabla^\perp\Delta \widetilde{\psi}_{\text{gap}} = d^{-2} \nabla^\perp \Delta \psi_R\big(d^{-1}(x-x_R)\big) = \nabla \widetilde{p}_{\text{gap}}(x).
    \end{equation*}
    By \cref{lem_expr_psiL_psiR}, for linear solid motion, $u_*^\perp = e_2$ or $u_*^\| = e_1$, the stream function takes the form $\psi_R(r,\theta)=rf_R(\theta)$.
    A direct computation in the polar basis $(e_{r_s}, e_{\theta_s})$ yields
    \begin{equation*}
        \Delta \widetilde{u}_{\text{gap}}(x) = -r_s^{-2}\begin{pmatrix}
            f_R'(\theta_s)+f_R'''(\theta_s)\\
            f_R(\theta_s) + f_R''(\theta_s)
        \end{pmatrix}.
    \end{equation*}
    Integrating the radial component of the pressure gradient yields
    \begin{equation*}
        \widetilde{p}_{\text{gap}}(x) = r_s^{-1}\big(f_R'(\theta_s) + f_R'''(\theta_s)\big) + p_{0,R}(\theta_s),
    \end{equation*}
    where $p_{0,R}$ depends only on $\theta_s$.
    Differentiating with respect to $\theta_s$ and equating the result to the angular component of $\Delta \widetilde{u}_{\text{gap}}$ yields 
    \begin{equation*}
        r_s^{-2}\big(f_R''(\theta_s)+f_R''''(\theta_s)\big) + r_s^{-1}p_{0,R}'(\theta_s) = -r_s^{-2}\big(f_R(\theta_s)+ f_R''(\theta_s)\big).
    \end{equation*}
    Since $f_R$ satisfies \cref{eq_fct_ang_edo}, it follows that $p_{0,R}'(\theta_s) = 0$, and hence $p_{0,R}$ is constant.
    An identical argument applies to the left subregion $A^d_L(\lambda d r_0,2)$.
    Setting the constants of integration $p_{0,L} = p_{0,R} = 0$, we obtain \cref{eq_1_th_pression} and \cref{eq_2_th_pression}.

    For the rotational solid motion $u_*^\circlearrowleft = x^\perp$, the stream function is $\psi_R^\circlearrowleft(r,\theta)=\frac{d}{2} r^2 f^\circlearrowleft_R(\theta)$.
    A direct computation in the polar basis $(e_{r_s}, e_{\theta_s})$ yields
    \begin{equation*}
        \Delta \widetilde{u}_{\text{gap}}^\circlearrowleft(x) = -\frac{1}{2r_s}\begin{pmatrix}
            4f_R^\circlearrowleft{}'(\theta_s)+f_R^\circlearrowleft{}'''(\theta_s)\\
            0
        \end{pmatrix}.
    \end{equation*}
    Integrating the radial component gives
    \begin{equation*}
        \widetilde{p}_{\text{gap}}^\circlearrowleft(x) = -\frac{1}{2}\big(4f_R^\circlearrowleft{}'(\theta_s)+f_R^\circlearrowleft{}'''(\theta_s)\big)\ln(r_s) + p_{0,R}(\theta_s),
    \end{equation*}
    Differentiating with respect to $\theta_s$, equating to the angular component of $\Delta \widetilde{u}_{\text{gap}}^\circlearrowleft$, and invoking \cref{eq_fct_ang_edo_rot}, we find $p_{0,R}'(\theta_s)=0$.
    Thus, $p_{0,R}$ is constant.
    A symmetric argument applies on $A^d_L(\lambda d r_0,2)$.
    Setting again $p_{0,L} = p_{0,R} = 0$ establishes \cref{eq_3_th_pression}.
\end{proof}

These explicit formulas imply that for any $x \in A^d_{L/R}(\lambda d r_0,2)$ satisfying $|x-x_{L/R}| > \mu_1 > 0$, the local pressure satisfies the bound
\begin{equation}\label{eq_bound_profil_p}
    |\widetilde{p}_{\text{gap}}| \leq C(\theta_L,\theta_R)\mu_1^{-1},
\end{equation}
where the logarithmic term in the rotational case is subdominant for small $\mu_1$.

Following the derivation of the leading-order velocity field, we use the local pressure $\widetilde{p}_\text{gap}$ to construct a global approximation.
Analogous to the stream function lifting in \cref{eq_def_psi_tilde_d}, we define the global pressure field $\widetilde{p}_d$ via
\begin{equation}\label{eq_exp_pd_glob}
    \widetilde{p}_d(x) :=\left[1-\chi_1\left(\frac{|x|}{dr_0}\right)\right]\chi_2(|x|)\widetilde{p}_\text{gap}.
\end{equation}
By construction, $\widetilde{p}_d$ is supported in $A^d(\lambda dr_0,2)$ and coincides with $\widetilde{p}_\text{gap}$ on the transition annulus $A^\text{tr} = A^d(dr_0,1)$.
\begin{proof}[Proof of \cref{eq_4_th_pression}]
    We first estimate the $L^2$-norm of the pressure on the transition region.
    Since $\widetilde{p}_d = \widetilde{p}_\text{gap}$ on $A^\text{tr} = A^\text{tr}_L \cup A^\text{tr}_R$, we have
    \begin{equation*}
        \|\widetilde{p}_d\|_{L^2(A^\text{tr})}^2 = \|\widetilde{p}_\text{gap}\|_{L^2(A^\text{tr}_L)}^2 + \|\widetilde{p}_\text{gap}\|_{L^2(A^\text{tr}_R)}^2.
    \end{equation*}
    To evaluate the integral over $A^\text{tr}_R$, we translate the origin to $x_R$ and partition the domain into a principal sector $A_\text{main} := \{(r,\theta) : r\in(dr_1,1),~\theta\in(0,\theta_R)\}$ and a remainder $A_\text{rem}$.
    For the linear solid motions $u^{\perp}_{*}$ and $u^{\|}_{*}$, a direct integration yields
    \begin{equation*}
        \int_{A_\text{main}}|\widetilde{p}_\text{gap}|^2 = \int^{\theta_R}_{0}\int_{dr_1}^{1}\left(\frac{\big(f_R'(\theta) + f_R'''(\theta)\big)}{r}\right)^2rdrd\theta = Q_R(\theta_R) \left|\log\left(d\right)\right|+\bigo(1),
    \end{equation*}
    where the explicit profile of $f_R$ from \cref{lem_expr_psiL_psiR} gives the geometric constant
    \begin{align*}
        Q_R(\theta_R) = 2B_R^2\big(\theta_R + \sin(\theta_R)\cos(\theta_R)\big) +4 B_RD_R\sin^2(\theta_R)+ 2D_R^2\big(\theta_R - \sin(\theta_R)\cos(\theta_R)\big) 
    \end{align*}
    Applying the parameter expressions from \cref{eq:coeffs_perp} and \cref{eq:coeffs_parallel}, we extract the leading-order coefficients
    \begin{equation*}
        Q_R^\perp(\theta_R) = 4B_R^\perp +2D_R^\| + 8 B_R^\perp D_R^\perp  ,\quad Q_R^\|(\theta_R)= -4D_R^\| -2B_R^\perp + 8 B_R^\| D_R^\|.
    \end{equation*}
    The remainder $A_\text{rem}$ consists of a $d$-scaled left component and a fixed-radius right component.
    Consequently, its energy contribution is uniformly bounded in $d$, which implies
    \begin{equation*}
        \|\widetilde{p}_\text{gap}\|_{L^2(A^\text{tr}_R)}^2 = Q_R(\theta_R) \left|\log\left(d\right)\right|+\bigo(1).
    \end{equation*}
    A symmetric argument on the left domain yields
     \begin{equation*}
        \|\widetilde{p}_\text{gap}\|_{L^2(A^\text{tr}_L)}^2 = Q_L(\theta_L) \left|\log\left(d\right)\right|+\bigo(1),
    \end{equation*}
    with the geometric constant
    \begin{align*}
        Q_L(\theta_L) = 2B_L^2\big(\theta_L + \sin(\theta_L)\cos(\theta_L)\big) -4 B_LD_L\sin^2(\theta_L)+ 2D_L^2\big(\theta_L - \sin(\theta_L)\cos(\theta_L)\big) 
    \end{align*}
    Substituting the explicit definitions for $B_L$ and $D_L$ gives
    \begin{equation*}
        Q_L^\perp(\theta_L) = -4B_L^\perp -2D_L^\| - 8 B_L^\perp D_L^\perp  ,\quad Q_L^\|(\theta_L)= 4D_L^\| + 2B_L^\perp - 8 B_L^\| D_L^\|.
    \end{equation*}
    Combining these estimates establishes the first two assertions of \cref{eq_4_th_pression}.

    For the rotational case, integrating over $A_\text{main}$ yields
    \begin{equation*}
        \int_{A_\text{main}}|\widetilde{p}_\text{gap}|^2 = \frac{1}{4} \int^{\theta_R}_{0}\int_{dr_1}^{1}\left(\big(f_R'(\theta) + f_R'''(\theta)\big)\ln(r)\right)^2rdrd\theta = \bigo(1).
    \end{equation*}
    Since the energy contribution from $A_\text{rem}$ is again bounded, and an analogous argument applies to $A^\text{tr}_L$, we obtain the third estimate of \cref{eq_4_th_pression}.

    To derive the final estimate, following \cref{lem_estim_ext_psid}, we decompose the integration domain $\supp(\widetilde{p}_d) \setminus A^\text{tr}$ into an outer annulus $A^\text{out}=A^d(1,2)$ and an inner annulus $A^\text{in}= A^d(\lambda d r_0,dr_0)$.
    Because the cutoff functions are uniformly bounded in $L^\infty$, it follows that
    \begin{equation*}
        \|\widetilde{p}_d\|^2_{L^2(\F^d\setminus A^\text{tr})} \leq C \bigg(\|\widetilde{p}_\text{gap}\|_{L^2(A^\text{in})}^2 + \|\widetilde{p}_\text{gap}\|_{L^2(A^\text{out})}^2\bigg).
    \end{equation*}

    Since both the $L^\infty(A^\text{out})$ norm of $\widetilde{p}_\text{gap}$ and the measure of $A^\text{out}$ are uniformly bounded in $d$, the integral over the outer annulus is $\bigo(1)$.
    In the inner annulus $A^\text{in}_{L/R} \setminus A^\text{tr}$, we have $|x-x_{L/R}| \geq c d$ for some constant $c>0$.
    Therefore, \cref{eq_bound_profil_p} yields the estimate $|\widetilde{p}_{\text{gap}}| \leq C(\theta_L,\theta_R)d^{-1} = \bigo(d^{-1})$.
    Combining this bound with $|A^\text{in}| = \bigo(d^2)$ ensures that the integral over $A^\text{in}$ is also $\bigo(1)$, which completes the proof.
\end{proof}
The key property for the proof of \cref{thm_pression} is the following, which rigorously quantifies the deviation of the approximate pair $(\widetilde{u}_d, \widetilde{p}_d)$ from being an exact Stokes pair.
\begin{property}[Approximate Stokes pair]\label{prop_approx_stokes_tilde}
    In the vanishing gap limit $d \to 0$, we have
    \begin{equation}
        \|-\Delta \widetilde{u}_d + \nabla \widetilde{p}_d\|_{H^{-1}(\F^d)} = \bigo(1).
    \end{equation}
\end{property}
\begin{proof}
    Let $R_d := -\Delta \widetilde{u}_d + \nabla \widetilde{p}_d$ denote the Stokes residual, which belongs to $[C^\infty(\F^d)]^2$ by construction. 
    For any test function $v \in [H^1_0(\F^d)]^2$, the action of the residual is given by
    \begin{equation*}
        \langle R_d, v \rangle_{H^{-1}, H^1_0} = \int_{\F^d} \nabla \widetilde{u}_d : \nabla v - \widetilde{p}_d \odiv(v).
    \end{equation*}

    Over the transition region $A^\text{tr}$, we have $\widetilde{p}_d = \widetilde{p}_\text{gap}$ and $\widetilde{u}_d = \widetilde{u}_\text{gap}$.
    By \cref{prop_pressure}, the local fields satisfy $-\Delta \widetilde{u}_\text{gap} + \nabla \widetilde{p}_\text{gap} = 0$.
    Integrating by parts over this subdomain yields
    \begin{equation*}
        \int_{A^\text{tr}} \nabla \widetilde{u}_d : \nabla v - \widetilde{p}_d \odiv(v) = \int_{\partial A^\text{tr}} (\partial_n \widetilde{u}_\text{gap}- \widetilde{p}_\text{gap} n)  \cdot v d\sigma,
    \end{equation*}
    where $n$ denotes the outward unit normal vector to $A^\text{tr}$.
    The boundary $\partial A^\text{tr}\cap\supp(v)$ consists of an outer macroscopic arc $\Gamma_\text{out}$ (where $|x|=1$, hence $|x-x_{L/R}| \sim 1$) and an inner microscopic arc $\Gamma_\text{in}$ (where $|x| = d r_0$, hence $|x-x_{L/R}| \sim d$).
    On $\Gamma_\text{out}$, the velocity gradient and pressure profiles are bounded independently of $d$ by \cref{eq_bound_profil_gap_psi}--\cref{eq_bound_profil_gap_psirot} and \cref{eq_bound_profil_p}, providing a uniform pointwise bound on the normal stress $\partial_n \widetilde{u}_\text{gap}- \widetilde{p}_\text{gap} n$.
    Because $\Gamma_\text{out}$ is away from the geometric singularity, standard trace estimates apply with a constant $C$ independent of $d$, yielding
    \begin{equation*}
        \left| \int_{\Gamma_\text{out}} (\partial_n \widetilde{u}_\text{gap} - \widetilde{p}_\text{gap} n) \cdot v \, d\sigma \right| \leq C \|v\|_{L^2(\Gamma_\text{out})} \leq C \|v\|_{H^1(\F^d)}.
    \end{equation*}
    On $\Gamma_\text{in}$, evaluating these profile estimates at $|x-x_{L/R}| \sim d r_0$ implies that the normal trace is bounded by $\bigo(d^{-1})$.
    Because the length of $\Gamma_\text{in}$ is $\bigo(d)$, its $L^2$-norm scales as
    \begin{equation*}
        \|\partial_n \widetilde{u}_\text{gap} - \widetilde{p}_\text{gap} n\|_{L^2(\Gamma_\text{in})} = \bigo(d^{-1/2}).
    \end{equation*}
    To estimate $v|_{\Gamma_\text{in}}$, let $y = x/d$ be the rescaled variable on the arc $\Gamma_\text{in}/d$, which has length $\bigo(1)$.
    Since $v$ vanishes on the Dirichlet boundary $\partial \Omega$, applying a Poincaré-trace inequality to $\hat{v}(y) = v(dy)$ yields $\|\hat{v}\|_{L^2(\Gamma_\text{in}/d)} \leq C \|\nabla \hat{v}\|_{L^2(\F^d/d)}$.
    Returning to the original coordinates, we obtain
    \begin{equation*}
        \|v\|_{L^2(\Gamma_{\text{in}})} \leq C d^{1/2} \|\nabla v\|_{L^2(\F^d)}.
    \end{equation*}
    The Cauchy--Schwarz inequality then yields the uniform bound
    \begin{equation*}
        \left| \int_{\Gamma_\text{in}} (\partial_n \widetilde{u}_\text{gap} - \widetilde{p}_\text{gap} n) \cdot v \, d\sigma \right| \leq C d^{-1/2} \|v\|_{L^2(\Gamma_\text{in})} \leq C \|v\|_{H^1(\F^d)}.
    \end{equation*}

    Finally, we bound the volume integral over the exterior domain $\F^d \setminus A^\text{tr}$ by
    \begin{equation*}
        \left| \int_{\F^d \setminus A^\text{tr}} \nabla \widetilde{u}_d : \nabla v - \widetilde{p}_d \odiv(v) \right| \leq \left( \|\nabla \widetilde{u}_d\|_{L^2(\F^d \setminus A^\text{tr})} + \|\widetilde{p}_d\|_{L^2(\F^d \setminus A^\text{tr})} \right) \|v\|_{H^1(\F^d)}.
    \end{equation*}
    The norms $\|\nabla \widetilde{u}_d\|_{L^2(\F^d \setminus A^\text{tr})}$ and $\|\widetilde{p}_d\|_{L^2(\F^d \setminus A^\text{tr})}$ are uniformly bounded by \cref{lem_estim_ext_psid} and \cref{eq_4_th_pression}, respectively.
    We conclude that $|\langle R_d, v \rangle| \leq C \|v\|_{H^1(\F^d)}$, which completes the proof.
\end{proof}

The proof of our main result relies on the Bogovskiǐ operator, a continuous right inverse of the divergence.
To fix notation, we recall its primary property and refer to \cite[Theorem IV.3.1]{BoyerFabrie2013} for a detailed proof.
\begin{property}[Bogovskiǐ Operator]\label{prop_cont_bog_gen}
    Let $\Omega\subset \R^N$ be a bounded, connected, and Lipschitz domain.
    There exists a continuous linear map $\Pi_\Omega:L^2_0(\Omega)\to [H^1_0(\Omega)]^N$ such that for any $f\in L^2_0(\Omega)$,
    \begin{equation*}
        \odiv(\Pi_\Omega(f))=f.
    \end{equation*}
\end{property}
\begin{corollary}\label{prop_bog_global}
    There exists a constant $C>0$, independent of $d$, such that for any $f\in L^2_0(\F^d)$,
    \begin{equation*}
        \|\Pi_{\F^d}f\|_{H^1_0(\F^d)} \leq C \|f\|_{L^2_0(\F^d)}.
    \end{equation*}
\end{corollary}
\begin{proof}[Sketch of proof]
    The complete technical proof is deferred to Appendix \ref{appendix_bog_const}.
    The main strategy is to cover the fluid domain $\F^d$ with three macroscopic subdomains, two shifted sectors forming the gap and an exterior domain.
    On each of these regions, the local Bogovskiǐ operator admits a continuity constant independent of $d$.
    By carefully decomposing $f$ into a sum of zero-mean components supported within these respective subdomains and applying the corresponding local operators, we obtain the uniform global bound.
\end{proof}
\begin{proof}[Proof of \cref{eq_5_th_pression}]
    Let $p^\text{err} := (p_d - \widetilde{p}_d) - \langle p_d - \widetilde{p}_d \rangle_{\F^d} \in L^2_0(\F^d)$.
    By \cref{prop_cont_bog_gen}, there exists a test function $w = \Pi_{\F^d}(p^\text{err}) \in [H^1_0(\F^d)]^2$ such that $\odiv w = p^\text{err}$ in $\F^d$.
    This gives
    \begin{equation*}
        \|p^\text{err}\|_{L^2_0(\F^d)}^2 = \int_{\F^d} p^\text{err} \odiv w = \langle \nabla(p_d - \widetilde{p}_d), w \rangle_{H^{-1}, H^1_0} \leq \|\nabla(p_d - \widetilde{p}_d)\|_{H^{-1}(\F^d)}\|w\|_{H^{1}_0(\F^d)}.
    \end{equation*}
    By \cref{prop_bog_global}, $w$ can be chosen so that $\|w\|_{H^1_0(\F^d)} \leq C \|p^\text{err}\|_{L^2_0(\F^d)}$, where $C > 0$ is a constant independent of $d$.
    Therefore,
    \begin{equation}\label{eq:grad_bound}
        \|p^\text{err}\|_{L^2_0(\F^d)} \leq C \|\nabla(p_d - \widetilde{p}_d)\|_{H^{-1}(\F^d)}.
    \end{equation}
    This reduces estimating $p^\text{err}$ in $L^2_0(\F^d)$ to bounding its gradient in $H^{-1}(\F^d)$, which is directly controlled via the Stokes system.

    Let $v \in [H^1_0(\F^d)]^2$.
    The weak formulation of the exact Stokes system \cref{eq_stokes} gives $\int_{\F^d} p_d \odiv v = \int_{\F^d} \nabla u_d : \nabla v$, so that
    \begin{equation*}
        \langle \nabla(p_d - \widetilde{p}_d), v \rangle_{H^{-1}, H^1_0} = \int_{\F^d} \nabla (u_d - \widetilde{u}_d) : \nabla v + \langle -\Delta \widetilde{u}_d + \nabla \widetilde{p}_d, v \rangle_{H^{-1}, H^1_0}.
    \end{equation*}
    The first term on the right-hand side is the velocity approximation error, while the second is a residual term measuring the deviation of $(\widetilde{u}_d, \widetilde{p}_d)$ from an exact Stokes pair.
    The first term is bounded using the Cauchy--Schwarz inequality and \cref{eq_5_th_stokes} on $A^\text{tr}$, and \cref{eq_6_th_stokes} and \cref{lem_estim_ext_psid} on $\F^d \setminus A^\text{tr}$.
    For the second term, \cref{prop_approx_stokes_tilde} immediately yields the desired continuity bound.
    Combining these two estimates gives $\|\nabla(p_d - \widetilde{p}_d)\|_{H^{-1}(\F^d)} = \bigo(1)$.
    Together with \cref{eq:grad_bound}, this concludes the proof.
\end{proof}

\section{Limitations of the variational reduction method}\label{sec_var}

In this section, we highlight why classical variational relaxation fails to capture the leading-order velocity field in Lipschitz geometries, despite its success for smooth or $C^{1,a}$ inclusions.
This breakdown justifies the stability-based approach developed in \cref{sec_laplace} and \ref{sec_stokes}.
For clarity, we restrict our analysis to the Laplace problem, which exhibits the same analytical difficulties as the Stokes problem.

The variational framework for constructing approximation profiles, as developed in \cite{GerardVaretHillairet2012, bonheure2024}, rests on an asymptotic separation of scales.
Consider a $C^{1,a}$ boundary, with $a \in (0,1]$, locally parameterized by $x_1 \mapsto d + |x_1|^{1+a}$.
It follows that the tangent at the point of minimal distance is horizontal.
Consequently, as the distance $d \to 0$, the characteristic horizontal scale in the near-contact zone dominates the vertical scale.
The gap becomes highly anisotropic, behaving as a narrow channel where vertical gradients strictly dominate horizontal ones.
This structural scale separation motivates the standard variational reduction.

By contrast, \cref{corol_error_estimate_var} illustrates that this dimensional reduction fails for $C^{0,1}$ Lipschitz inclusions.
A sharp tip, parameterized locally by $x_1 \mapsto d + \alpha|x_1|$ for some $\alpha > 0$, maintains a constant, non-vanishing slope.
It follows that the horizontal and vertical distances to the singularity scale identically.
The near-contact geometry thus remains isotropic rather than degenerating into a thin tubular domain.
Without a single dominant direction, there is no justification for reducing the gradient, rendering variational relaxation ineffective.

To simplify the computations, we assume a symmetric tip.
Specifically, in a neighborhood of the origin of order one, the boundary $\partial S^d$ coincides with the graph of $x_1 \mapsto d + \alpha |x_1|$ for some slope $\alpha > 0$.
In the notation of the introduction, this corresponds to setting $\theta_L = \theta_R = \theta_\alpha$, where the aperture angle is $\theta_\alpha = \arctan(\alpha) \in [0, \pi/2)$.
To define the near-contact geometry, we introduce the Cartesian counterpart of the polar gap domain $G^d(\rho)$, given by
\begin{equation*}
    \G^d(\rho) = \big\{(x_1,x_2) \mid x_1 \in (-\rho,\rho),~ x_2 \in (0, d+\alpha |x_1|) \big\}.
\end{equation*}

We first recall the standard variational formulation of the Laplace problem \cref{eq_lapl}. Defining the admissible class $A_{\Omega} = \big\{\phi \in H^1_0(\Omega) \mid \phi = 1 \text{ in } S^d\big\}$, the potential $\phi_d$ is the unique minimizer of
\begin{equation}\label{eq:var_lapl_comp}
    \phi_d = \argmin_{\phi\in A_{\Omega}} \int_{\Omega} |\nabla \phi|^2.
\end{equation}
Motivated by the structural scale separation inherent to smooth inclusions, one anticipates that the dominant contribution to the Dirichlet energy arises from the vertical derivative within the gap vicinity, namely $\G^d(1)$.
Under this ansatz, a natural candidate to approximate the exact solution is $\widehat{\phi}_d$, defined as the unique minimizer of the relaxed problem
\begin{equation}\label{eq_relax_local_formul_lapl}
    \widehat{\phi}_d =\argmin_{\phi \in A_{\G^d(1)}}\int_{\G^d(1)} |\partial_{x_2} \phi|^2,
\end{equation}
where the localized admissible class is given by $A_{\G^d(1)}=\big\{\phi \in H^1(\G^d(1)) \mid  \phi= 1 \text{ on } \partial S^d\cap\partial\G^d(1) ,~ \phi= 0 \text{ on } \partial \Omega\cap\partial\G^d(1)\big\}$.
The classical relaxation method relies on the property that integrating this localized problem yields an explicit candidate to capture the singular behavior of $\phi_d$.
\begin{proposition}[Laplace solution approximation profile by variational method]\label{thm_lapl_var}
    The reduced functional method yields the approximation profile $\widehat{\phi}_d$, defined on $\G^d(1)$ by
    \begin{equation}
        \widehat{\phi}_d(x_1,x_2) = \frac{x_2}{d+\alpha|x_1|}.
    \end{equation}
\end{proposition}
\begin{proof}
    To solve the relaxed problem \cref{eq_relax_local_formul_lapl}, we perform a vertical rescaling by introducing the relative vertical coordinate $r(x_1,x_2) = \frac{x_2}{d+\alpha|x_1|}$.
    Applying the change of variables $(x_1,x_2)\mapsto(x_1,r)$, the functional becomes
    \begin{equation*}
        \int_{\G^d(1)} |\partial_{x_2} \phi|^2 = \int_{-1}^{1}\int_0^{1}\frac{|\partial_r \widetilde{\phi}(x_1,r)|^2}{d+\alpha|x_1|} dr\,dx_1,
    \end{equation*}
    where $\widetilde{\phi}(x_1,r) := \phi(x_1, (d+\alpha|x_1|)r)$.
    Minimizing this integral reduces to the pointwise minimization of $\int_0^1 |\partial_r \widetilde{\phi}(x_1,r)|^2 \,dr$ for almost every $x_1 \in (-1,1)$.
    The associated Euler--Lagrange equation $\partial_{rr} \widetilde{\phi} = 0$, combined with the boundary conditions $\widetilde{\phi}(x_1,0)=0$ and $\widetilde{\phi}(x_1,1)=1$, yields $\widetilde{\phi}(x_1,r) = r$.
    Returning to the original coordinates, the local minimizer is given by
    \begin{equation*}
        \widehat{\phi}_d(x_1,x_2) := r(x_1,x_2) = \frac{x_2}{d+\alpha|x_1|}.
    \end{equation*}
\end{proof}
By comparing this candidate to the exact asymptotic profile $\widetilde{\phi}_d$ established in \cref{thm_stokes}, we can rigorously quantify the breakdown of the relaxation method.

\begin{corollary}[Variational breakdown]\label{corol_error_estimate_var}
    For any slope $\alpha > 0$, as $d\to 0$,
    \begin{equation*}
        \|\widetilde{\phi}_d- \widehat{\phi}_d\|^2_{\dot{H}^1(A^\text{tr})} = 2\left[\frac{1}{3}\alpha   + \alpha^{-1} - \theta_\alpha^{-1}\right]|\log(d)| + \bigo(1).
    \end{equation*}
\end{corollary}
Notably, this result establishes that the error associated with the variational candidate $\widehat{\phi}_d$ shares the logarithmic divergence rate of the leading-order energy.
Consequently, the reduced functional method fails to capture the dominant behavior of the solution in Lipschitz geometries.

\begin{remark}
    The structural limitation of the variational ansatz is explicitly visible at the gradient level. 
    Shifting the origin to the pole $x_\alpha=-\sgn(x_1)\frac{d}{\alpha}e_1$ allows us to compare the gradient of the variational candidate $\widehat{\phi}_d$ with the exact leading-order profile $\widetilde{\phi}_d$ from \cref{thm_stokes}:
    \begin{equation}\label{eq_grad_approx_profiles}
        \nabla \widetilde{\phi}_d(x) = \frac{\sgn(x_1)}{\theta_\alpha|x-x_\alpha|^2}(x-x_\alpha)^\perp, \qquad \nabla \widehat{\phi}_d(x) =\frac{\sgn(x_1)}{\alpha\big(|x-x_\alpha|^2 - x_2^2\big)}(x-x_\alpha)^\perp.
    \end{equation}
    The $-x_2^2$ term in the denominator of $\nabla \widehat{\phi}_d$ reduces the radial metric to its horizontal component.
    This implicitly enforces a thin-film scaling, $|x_2| \ll |x_1 - x_{\alpha}\cdot e_1|$, which fails for isotropic Lipschitz tips.
\end{remark}

This result clarifies a fundamental limitation of classical relaxation arguments.
While this 1D thin-film reduction is asymptotically valid for $C^{1,a}$ boundaries  with $a \in [0, 1)$, where the tangent becomes horizontal at the origin and the scales decouple, it is fundamentally incompatible with the isotropic scaling of a Lipschitz tip, where $x_1$ and $x_2$ remain of the same order. 
This establishes that gradient reduction is strictly a consequence of boundary smoothness rather than mere gap narrowness, fully justifying the necessity of the stability-based framework developed in the subsequent sections.

\begin{remark}
    As $\alpha \to 0$, the tip artificially flattens and restores the anisotropic thin-film scaling, which explains why the relative error vanishes and the variational method successfully recovers the leading-order behavior in this specific regime.
\end{remark}

\begin{proof}[Proof of \cref{corol_error_estimate_var}]
    Let $\widetilde{\phi}_d$ and $\widehat{\phi}_d$ be the profiles from \cref{thm_lapl} and \cref{thm_lapl_var}, and let $x_\alpha = -\sgn(x_1)\frac{d}{\alpha}e_1$ (corresponding to $x_L$ and $x_R$).
    From \cref{eq_grad_approx_profiles}, it follows that
    \begin{equation*}
        \|\widetilde{\phi}_d - \widehat{\phi}_d \|^2_{\dot{H}^1(A_R^\text{tr})} = \int_{A_R^\text{tr}-x_\alpha}\left[\frac{1}{\theta_\alpha|x|^2} - \frac{1}{\alpha|x_1|^2}\right]^2|x|^2dx.
    \end{equation*}
    We again partition the domain as $A_R^\text{tr}-x_\alpha = A_\text{main} \cup A_\text{rem}$ where $A_\text{main} = \{(r,\theta) \mid r \in (dr_1, 1), \theta \in (0,\theta_R)\}$ is the principal angular sector, and $r_1>0$ depends only on $\alpha$.
    Thus,
    \begin{equation*}
        \int_{A_\text{main}}\bigg[\frac{1}{\theta_\alpha|x|^2} - \frac{1}{\alpha|x_1|^2}\bigg]^2|x|^2 
        = \left[\frac{1}{\theta_\alpha} - \frac{2\tan(\theta_\alpha)}{\alpha\theta_\alpha} + \frac{\tan(\theta_\alpha)+\frac{1}{3}\tan^3(\theta_\alpha)}{\alpha^2}\right]|\log(dr_1)|.
    \end{equation*}
    Because $A_\text{rem}$ is confined to a fixed-radius outer region and a bounded $d$-scaled inner region, its associated contribution is controlled independently of $d$.
    Then,
    \begin{equation*}
        \|\widetilde{\phi}_d - \widehat{\phi}_d \|^2_{\dot{H}^1(A_R^\text{tr})} = \left[\frac{1}{\theta_\alpha} - \frac{2\tan(\theta_\alpha)}{\alpha\theta_\alpha} + \frac{\tan(\theta_\alpha)+\frac{1}{3}\tan^3(\theta_\alpha)}{\alpha^2}\right]|\log(dr_1)| + \bigo(1).
    \end{equation*}
    Recalling the geometric relation $\alpha = \tan(\theta_\alpha)$ and extending the estimate to the full transition gap by symmetry, we deduce the expected relation.
\end{proof}

\appendix
\renewcommand{\addcontentsline}[3]{} 

\section{Uniform continuity of the Bogovskiǐ operator}\label{appendix_bog_const}

\begin{proof}[Proof of \cref{prop_bog_global}]
    Let us introduce the sectors
    \begin{equation*}
        C_L = \{(r,\theta)\mid r<1, ~\theta\in(\pi-\theta_L,\pi)\}, \qquad C_R = \{(r,\theta)\mid r<1, ~\theta\in(0,\theta_R)\},
    \end{equation*}
    and their shifted counterparts $S_{L/R}=C_{L/R}+x_{L/R}$, which together form the tip $T^d(1)=S_L\cup S_R$.
    Observe that $T^d(1)$ differs from $G^d(1)$ because the vertices of the sectors are not located at the origin.
    Since $C_L$ and $C_R$ are independent of $d$, their associated Bogovskiǐ operators $\Pi_{C_{L/R}}$ admit continuity constants independent of $d$.
    By translation invariance, the operators $\Pi_{S_{L/R}}$ on the shifted domains inherit these same uniform bounds.
    Next, define the macroscopic exterior domain $\F^d_\text{ext} = \F^d \setminus \overline{T^d(1/2)}$. Excluding the singular tip region ensures that $\F^d_\text{ext}$ is a regular perturbation of a fixed Lipschitz domain.
    Consequently, the associated operator $\Pi_{\F^d_\text{ext}}$ also admits a uniform continuity constant.

    Introduce the mean angle $\theta_M = \frac{1}{2}(\pi - \theta_L + \theta_R)$ along with its associated intersection point $x_M = d y_M$ where $y_M = (-\cot(\theta_M), 0)$.
    These quantities are illustrated in \cref{fig:Ex_polygon_plane_sys}.
    We partition the domain $\F^d$ as
    \begin{equation*}
        E_{L/R} = S_{L/R} \cap \{\arg(x - x_M) \gtrless \theta_M\},\qquad E_{ext} = \F^d \setminus \overline{(E_L \cup E_R)}.
    \end{equation*}
    These sets are disjoint and satisfy $E_{L/R} \subset S_{L/R}$ and $E_{\text{ext}} \subset \F^d_{\text{ext}}$ (given $T^d(1/2) \subset E_L \cup E_R$).

    As the restrictions $f\mathds{1}_{E_L}$ and $f\mathds{1}_{E_R}$ do not generally have zero mean, they cannot serve as direct arguments for the Bogovskiǐ operators.
    We must therefore adjust their averages.
    By construction, the sets $V^d_{L/R} := E_{L/R} \cap \F^d_\text{ext} \subset A^d_{L/R}(1/2,1)$ lie in the macroscopic region.
    Since $V^d_{L/R}$ converges to a fixed, non-empty annular sector as $d \to 0$, we may choose smooth cutoff functions $\omega_{L/R}$ supported in a $d$-independent compact subset $V^0_{L/R} \subset V^d_{L/R}$, such that $\int_{V^0_{L/R}} \omega_{L/R} = 1$ and $\|\omega_{L/R}\|_{L^2(V^0_{L/R})}$ is bounded independently of $d$.
    For any $f \in L^2_0(\F^d)$, we decompose $f$ into zero-mean components
    \begin{align*}
        f_{L/R} = f \mathds{1}_{E_{L/R}} - \left(\int_{E_{L/R}} f\right) \omega_{L/R},\qquad
        f_\text{ext} = f \mathds{1}_{E_{ext}} + \left(\int_{E_L} f\right) \omega_L + \left(\int_{E_R} f\right) \omega_R.
    \end{align*}
    Summing these three terms recovers exactly $f$.
    By construction, $f_{L/R} \in L^2_0(S_{L/R})$.
    Furthermore, since $f$ has zero mean over $\F^d$, it follows that $f_\text{ext} \in L^2_0(\F^d_\text{ext})$.

    Applying the Cauchy-Schwarz inequality yields
    \begin{equation*}
        \|f_L\|_{L^2(S_L)} \leq \|f\|_{L^2(E_L)} + |E_L|^{1/2} \|f\|_{L^2(E_L)} \|\omega_L\|_{L^2(V^0_{L})} \leq C \|f\|_{L^2(\F^d)},
    \end{equation*}
    for some constant $C > 0$ independent of $d$.
    Analogous uniform bounds hold for $\|f_R\|_{L^2}$ and $\|f_\text{ext}\|_{L^2}$.
    Extending $\Pi_{S_L}(f_L)$, $\Pi_{S_R}(f_R)$, and $\Pi_{\F^d_\text{ext}}(f_\text{ext})$ by zero outside their respective domains yields functions in $H^1_0(\F^d)$.
    We then construct a global right inverse for the divergence equation over $\F^d$ by
    \begin{equation*}
        u = \Pi_{S_L}(f_L) + \Pi_{S_R}(f_R) + \Pi_{\F^d_\text{ext}}(f_\text{ext}).
    \end{equation*}
    Combining the uniform continuity of the local operators with the $L^2$-bounds on the zero-mean components yields
    \begin{equation*}
        \|u\|_{H^1_0(\F^d)} \leq C\big( \|f_L\|_{L^2(S_L)} +  \|f_R\|_{L^2(S_R)} + \|f_\text{ext}\|_{L^2(\F^d_\text{ext})} \big)\leq C \|f\|_{L^2(\F^d)},
    \end{equation*}
    completing the proof.
\end{proof}

\printbibliography
\end{document}